\documentclass[12pt]{article}
\usepackage[margin=1.0in]{geometry}
\usepackage[utf8]{inputenc}
\usepackage{tikz}
\usepackage{mathtools}
\usepackage{mathrsfs,amssymb,amsfonts,amsthm,amsmath,amsbsy}
\usepackage{inputenc}
\usepackage[english]{babel}
\usepackage[T1]{fontenc}
\usepackage{hyperref}
\usepackage{float}
\usepackage{ upgreek }
\usepackage{bm}
\usepackage{amssymb}
\usepackage{mdframed}
\usepackage{verbatim}
\usepackage{dsfont}
\usepackage{stmaryrd}
\usepackage{soul}
\usepackage{titlesec} 
\usepackage{listings}
\usepackage{lipsum}
\usepackage[export]{adjustbox}
\usepackage[font={small}, labelfont={rm,bf}, margin=1cm]{caption}
\usepackage[
backend=biber,
style=alphabetic,
citestyle=ieee-alphabetic,
maxnames=10
]{biblatex}
\renewbibmacro{in:}{%
  \ifentrytype{article}{}{\printtext{\bibstring{in}\intitlepunct}}}
\newtheorem{theorem}{Theorem}[section]
\newtheorem{proposition}[theorem]{Proposition} 
\newtheorem{corollary}[theorem]{Corollary}
\newtheorem{conjecture}[theorem]{Conjecture}
\newtheorem{openquestion}[theorem]{Open question}
\newtheorem{lemma}[theorem]{Lemma}

\newtheorem{definition1}[theorem]{Definition}

\newtheorem{remark1}[theorem]{Remark}

\newcommand{\somme}[3]{\sum_{#1}^{#2} #3}
\newcommand{\indi}[1]{\mathds{1}_{#1}}
\newcommand{\Pf}{\mathbb{P}}

\newcommand{\inte}[4]{\int_{#1}^{#2} #3 \, \mathrm{d}#4}

\newmdenv[
  roundcorner=5pt, 
  frametitlebackgroundcolor=gray!0, 
  frametitlerule=false, 
  frametitlerulewidth=1pt, 
  frametitlerulecolor=black, 
  frametitlefont=\bfseries\large, 
  backgroundcolor=gray!0, 
  linewidth=1.5pt, 
  innertopmargin=0pt, 
  innerbottommargin=10pt, 
  innerleftmargin=10pt, 
  innerrightmargin=10pt, 
  frametitlealignment=\centering, 
  topline=false, 
  rightline=false, 
  bottomline=false, 
  leftline=true, 
    leftmargin=-0.3cm, 
  rightmargin=0cm
]{boxproof}

\author{Timothy Budd\footnote{IMAPP, Radboud University, Nijmegen, t.budd@science.ru.nl} \hspace{0.3cm}and \hspace{0.15cm} Tanguy Lions\footnote{ENS Lyon, tanguy.lions@ens-lyon.fr}}

\date{}
\title{\textbf{The tight length spectrum of large-genus random hyperbolic surfaces with many cusps}}
\begin{document}

\maketitle
\begin{abstract}
    Since the work of Mirzakhani \& Petri \cite{Mirzakhani_petri_2019} on random hyperbolic surfaces of large genus, length statistics of closed geodesics have been studied extensively. We focus on the case of random hyperbolic surfaces with cusps, the number $n_g$ of which grows with the genus $g$. We prove that if $n_g$ grows fast enough and we restrict attention to special geodesics that are \emph{tight}, we recover upon proper normalization the same Poisson point process in the large-$g$ limit for the length statistics. The proof relies on a recursion formula for tight Weil-Petersson volumes obtained in \cite{budd2023topological} and on a generalization of Mirzakhani's integration formula to the tight setting.
\end{abstract}
\begin{center}
   \begin{figure}[H]
    \centering
\includegraphics[scale = 0.45]{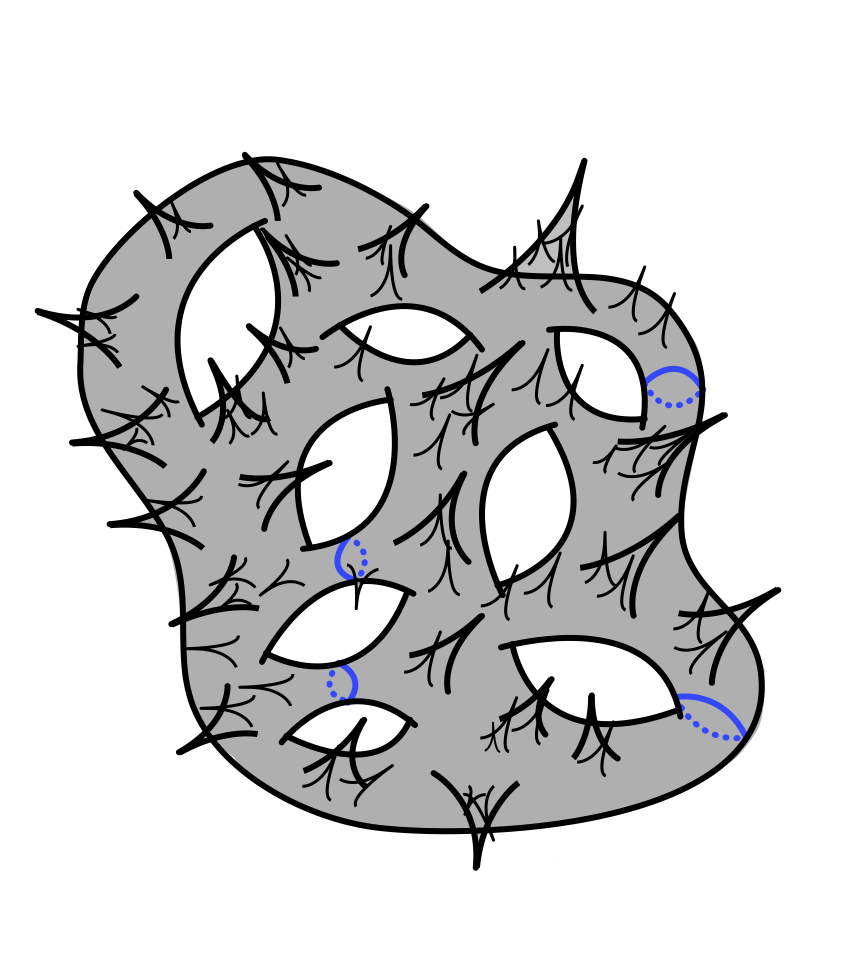}
\caption{A hyperbolic surface of genus $8$ with many cusps. The blue curves are the $4$ shortest \emph{tight} geodesics on the surface.}
    \label{fig:enter-label}
\end{figure} 
\end{center}
\maketitle
\section{Introduction}
\paragraph{Random hyperbolic surfaces.}
Random hyperbolic surfaces of large genus have been extensively studied recently (see~\cite{Magee_2022,Toward_spectral_gap,Wu_2022,Mirzakhani_petri_2019}). These advances have been made possible thanks to a good understanding of the Weil-Petersson volumes $V_{g,n}(\mathbf{L})$ of the moduli space of genus-$g$ hyperbolic surfaces with $n$ geodesic boundaries of length $\mathbf{L} = (L_1,\cdots,L_n) \in \mathbb{R}_{\ge 0}^n$ as $g \to +\infty$. For fixed genus $g$, there are few works studying the case of random hyperbolic surfaces of genus $g$ with $n$ cusps (see~\cite{hide2024largen}). The recent work \cite{budd2023topological} by Zonneveld and one of the authors provides a new framework to deal with the situation where the number of cusps is random. For $\mu \ge 0$, denote $F_{g}(\mu) = \somme{n=0}{\infty}{\frac{\mu^n}{n!}V_{g,n}}$, where $V_{g,n} = V_{g,n}(0)$. 
By \cite[Theorem~6.1]{manin1999invertible}, the generating function $F_{g}(\mu)$ is finite if and only if $\mu \leq \mu_c$ if $g=0$ or $\mu < \mu_c$ if $g\ge 1$, where $\mu_c\coloneqq \frac{j_0J_1(j_0)}{4\pi^2}=0.0316\ldots$ and $j_0$ is the first zero of the Bessel function of the first kind $J_0$. For $\mu < \mu_c$, we define $\Pf^{\mathrm{WP}}_{g,\mu}$, the \emph{$\mu$-Boltzmann probability measure} on hyperbolic surfaces of genus $g$, by first sampling $n$ with probability $\displaystyle F_{g}(\mu)^{-1}\frac{\mu^n}{n!}V_{g,n}$ and then chosing a surface $X$ under $\Pf^{\mathrm{WP}}_{g,n}$, the Weil-Petersson probability measure on hyperbolic surfaces of genus $g$ with $n$ cusps. This Boltzmann model is natural from the point of view of statistical physics. Also, from a probabilistic point of view, the model has better independence properties. We denote by $\mathcal{N}_{g,\mu}$ the number of cusps chosen.

\paragraph{Statistics of lengths of closed geodesics.}
 A closed geodesic $\gamma$ is said to be \emph{primitive} if it is not obtained by iterating a simple closed geodesic two or more times. The \emph{length spectrum} of a hyperbolic surface $X$ is the list of the ordered lengths $\ell_1 \le \ell_2 \le \cdots$ of closed geodesics on $X$. In the same way, we define the primitive length spectrum by considering only primitive closed geodesics. We denote by $\Lambda_g$ the random multiset of lengths of primitive closed geodesics of a surface $X$ sampled under $\Pf^{\mathrm{WP}}_g$. The recent work of Mirzakhani and Petri \cite{Mirzakhani_petri_2019} gives a full description of the large-genus asymptotic behaviour of the length spectrum for short geodesics.
 \begin{theorem}[Theorem~4.1 of \cite{Mirzakhani_petri_2019}]\label{Mirzakhani_Petri}\label{Length_spectrum_without_cusps}
     As $g \to \infty$, we have the convergence in distribution
     \begin{align*}
         \Lambda_g \xrightarrow[g \to \infty]{\mathrm{(d)}} \mathcal{P},
     \end{align*}
    where $\mathcal{P}$ is a Poisson point process on $\mathbb{R}_{+}$ with intensity $\frac{\cosh{t}-1}{t}\mathrm{d}t$. In the convergence, the random multiset $\Lambda_{g}$ is regarded as a random point process on $[0,\infty)$.
 \end{theorem}
Janson and Louf proved in \cite{janson2021unicellular} a similar result for another model of large genus random geometry: unicelullar maps. Recently, this result has been extended to the case of random metric maps \cite{barazer2023lengt}.

\paragraph{Main results.}We aim to extend the theorem to the case where the number of cusps $n_g$ grows with the genus $g$. More precisely, we are interested in the regime where $g = o(n_g)$. For technical reasons our proof will be in the regime $n_g \gg g^3$. This regime is very different from the one in which $n_g=0$ and requires to modify the statement of the theorem for several reasons:
\begin{itemize}
    \item[$\bullet$] The work \cite{budd2025random} of one of the authors and Curien proves that the metric of the Weil-Petersson random hyperbolic surface of genus $0$ with $n$ cusps normalized by $n^{-1/4}$ converges in distribution to (a constant multiple of) the Brownian sphere as $n\to\infty$. This establishes a connection with random planar maps, where similar results have been obtained in the last decade (see~\cite{Miermont_quadrangulation,Le_Gall_uniqueness,Marzouk_2018}). It is natural to expect an analogue result for fixed genus $g \ge 1$, in which case the limiting object is a Brownian surface of genus $g$. This scaling limit has been established in the case of uniform quadrangulations in \cite{bettinelli2022compact}. Contrary to random hyperbolic surfaces of genus $g$ chosen under $\Pf_{g}^{\mathrm{WP}}$,  Brownian surfaces have a fractal structure (see Figure~\ref{diff_no_cusps_a_lot_of_cusps}). We should expect characteristics of these Brownian surfaces to appear in our regime as $g\to\infty$ and therefore we need to deal with this fractal structure.
    \item[$\bullet$] When the number of cusps is large, we may encounter collections of geodesics of almost equal length that are not homotopic because of cusps separating them. Thus, in order to have a chance of obtaining a Poisson point process in the limit we should focus on a subclass of geodesics.
    \item[$\bullet$] Typical distances on random surfaces grow with increasing number of cusps, so the lengths of geodesics  should be appropriately normalized.
\end{itemize}
\begin{figure}[H]
    \centering
    \includegraphics[scale = 0.5]{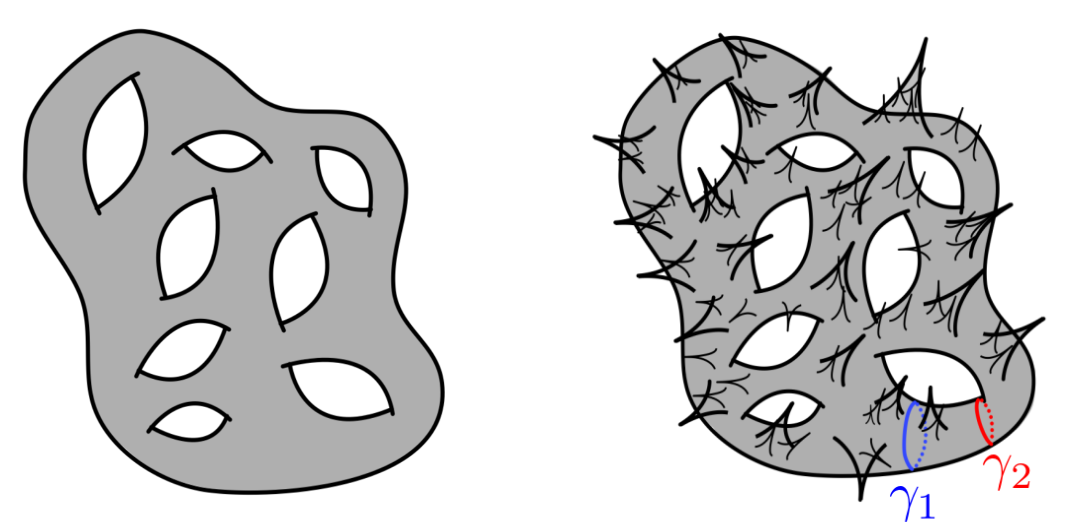}
    \caption{On the left a surface chosen under $\Pf^{\mathrm{WP}}_{8}$ and on the right a surface chosen under $\Pf^{\mathrm{WP}}_{8,74}$. On the right, the two geodesics $\gamma_1$ and $\gamma_2$ are very close in length, but not homotopic due to the cusps that separate them. The presence of a large number of cusps entails a fractal structure for the right surface which is not the case on the left.}
    \label{diff_no_cusps_a_lot_of_cusps}
\end{figure}
We restrict our study to a subset of geodesics called \emph{tight closed geodesics}. Informally, a closed geodesic is said to be tight (see Section~\ref{preliminaries} for a precise definition) if one cannot shorten the curve with a continuous deformation, even when allowing the curve to pass over cusps. Denote by $\Lambda^{\mathrm{tight}}_{g,\mu}$ the multiset of lengths of primitive tight closed geodesics of a surface sampled under $\Pf^{\mathrm{WP}}_{g,\mu}$. The following theorem proves that in the large genus regime, when the parameter $\mu_g$ is approaching the critical value $\mu_c$ sufficiently fast, we recover in the limit the same Poisson point process as in Theorem~\ref{Mirzakhani_Petri}.
\begin{theorem}\label{main_theorem}
    For any $\mu_{g} \xrightarrow[g \to \infty]{}\mu_c$ such that $\mu_c -\mu_g = o(g^{-2})$, we have
     \begin{align*}
         \displaystyle \alpha_1 (\mu_c- \mu_g)^{\frac{1}{4}}\Lambda^{\mathrm{tight}}_{g,\mu_g} \xrightarrow[g \to \infty]{\mathrm{(d)}} \mathcal{P},
     \end{align*}     
     where $\mathcal{P}$ is a Poisson point process on $\mathbb{R}_{+}$ with intensity $\frac{\cosh{t}-1}{t}\mathrm{d}t$ and $\alpha_1 = \sqrt{\frac{6}{\pi}\sqrt{\frac{2j_0}{J_1(j_0)}}} = 2.41105\ldots$ where $j_0$ is the first zero of the Bessel function of the first kind $J_0$. In the convergence, the random multiset $\Lambda^{\mathrm{tight}}_{g,\mu_g}$ is regarded as a random point process on $[0,\infty)$.
\end{theorem}
\noindent
This convergence in distribution can be equivalently stated as follows. 
Consider the random variable $N_{g,\mu,a,b}^{\mathrm{tight}}$ for $0\leq a < b$ counting the number of primitive tight closed geodesics on a hyperbolic surface $X$ chosen under $\Pf^{\mathrm{WP}}_{g,\mu}$ with length in the interval $[\alpha_1^{-1}(\mu_c-\mu)^{-\frac{1}{4}}a,\alpha_1^{-1}(\mu_c-\mu)^{-\frac{1}{4}}b]$.
Then for any $\mu_{g} \underset{g \to +\infty}{\rightarrow}\mu_c$ such that $\mu_c -\mu_g = o(g^{-2})$, and any $0 \le a_1 < b_1 < \cdots < a_r < b_r$, we have
\begin{align*}
 \left(N^{\mathrm{tight}}_{g,\mu_g,a_i,b_i}\right)_{1\le i \le r} \xrightarrow[g \to \infty]{\mathrm{(d)}}(\mathcal{P}_i)_{1 \le i \le r},
\end{align*}     
where $(\mathcal{P}_i)_{1 \le i \le r}$ is a family of independent Poisson variables with means $\lambda_{a_i,b_i}=\inte{a_i}{b_i}{\frac{\cosh{t}-1}{t}}{t}$. 

The following is a slight modification of Theorem~\ref{main_theorem} that makes it more explicit that the shortest tight geodesics have length of order $(n/g)^{1/4}$.
\begin{corollary}\label{main_corollary}
        For any sequence $(n_g)_{g \ge 0}$ such that $n_{g}/g^3 \to\infty$, there exists a sequence $(\mu_g)_{g \ge 0}$ such that $\mu_g \to \mu_c$ and $\mathcal{N}_{g,\mu_g} / n_g \xrightarrow[g \to \infty]{(\mathbb{P})}1$. Furthermore
   \begin{align*}
         \displaystyle \alpha_2 \bigg(\frac{n_g}{g}\bigg)^{-\frac{1}{4}}\Lambda^{\mathrm{tight}}_{g,\mu_g} \xrightarrow[g \to \infty]{\mathrm{(d)}} \mathcal{P},
     \end{align*}     
     where $\mathcal{P}$ is a Poisson point process on $\mathbb{R}_{+}$ with intensity $\frac{\cosh{t}-1}{t}\mathrm{d}t$ and $\alpha_2 = \frac{1}{\pi}\sqrt{3j_0\sqrt{5}}=1.27848\ldots$ where $j_0$ is the first zero of the Bessel function of the first kind $J_0$.  In the convergence, the random multiset $\Lambda_{g,\mu_g}$ is regarded as a random point process on $[0,\infty)$.
\end{corollary}
\noindent The systole $\ell_{\mathrm{sys}}$ of a hyperbolic surface is the length of its shortest non-trivial closed geodesic. 
Similarly, the non-separating systole $\ell^{\,\mathrm{ns}}_{\mathrm{sys}}$ is the length of the shortest non-separating geodesic.
The latter geodesic is necessarily tight, whereas the former is tight unless it bounds a disk with two or more cusps.
Therefore it is natural to also define the tight systole $\ell^{\,\mathrm{tight}}_{\mathrm{sys}}$ to be the length of the shortest (not necessarily separating) tight geodesic.

\begin{corollary}\label{systole}
    Let $(n_g)_{g \ge 0}$ and $(\mu_g)_{g \ge 0}$ be sequences as in Corollary~\ref{main_corollary}. 
    Then
    \begin{align*}
        \Pf^{\mathrm{WP}}_{g,\mu_g}(\ell^{\,\mathrm{ns}}_{\mathrm{sys}} = \ell^{\,\mathrm{tight}}_{sys}) \xrightarrow[g \to \infty]{}1.
    \end{align*}
Moreover, under $\Pf^{\mathrm{WP}}_{g,\mu_g}$, the random variables $\alpha_2(\frac{n_g}{g})^{-\frac{1}{4}}\ell^{\,\mathrm{ns}}_{\mathrm{sys}}$ and $\alpha_2(\frac{n_g}{g})^{-\frac{1}{4}}\ell^{\,\mathrm{tight}}_{\mathrm{sys}}$, with $\alpha_2$ as in Corollary~\ref{main_corollary}, both converge in distribution to a random variable $X$ with law
\begin{align*}
    \Pf(X \ge t) = \exp\bigg(-\inte{0}{t}{\frac{\cosh{t}-1}{t}}{t}\bigg).
\end{align*}
\end{corollary}
\begin{remark1}
    We expect $\alpha_2\left(\frac{n_g}{g}\right)^{-\frac{1}{4}}\mathbb{E}\left[\ell^{\mathrm{tight}}_{sys}\right] \xrightarrow[g\to\infty]{} \inte{0}{\infty}{\exp(-\lambda_{0,t})}{t}$. However, this convergence requires extra control on $\Pf^{\mathrm{WP}}_{g,\mu_g}\bigg(\ell^{\mathrm{tight}}_{sys} \ge \alpha_2(\frac{n_g}{g})^{\frac{1}{4}}t \bigg)$, which we currently lack.
\end{remark1}
\paragraph{Open questions and conjecture.}
We state several questions that remain open and are natural extensions of our results.
\begin{openquestion}
    Does Theorem~\ref{main_theorem} hold for every sequence  $\mu_g \xrightarrow[g\to\infty]{}\mu_c$? Is there still a limiting point process if instead $\mu_g \to c \in [0,\mu_c)$? Is it possible to extend Theorem~\ref{main_theorem} under the probability $\Pf^{\mathrm{WP}}_{g,n_g}$ instead of $\Pf^{\mathrm{WP}}_{g,\mu_g}$? 
\end{openquestion}
\noindent
Our results motivate the conjecture that the limiting Poisson point process is also the large genus limit of the geodesic length spectrum of the Brownian surface of genus $g$ as $g \to \infty$.
\begin{conjecture}\label{length_spectrum_brownian_surface}
Let $\mathbf{S}_{g}$ be the unit-area Brownian surface of genus $g$ with no boundary introduced in \cite[Theorem~1.1]{bettinelli2022compact}. Define $\Upxi_g$ to be the multiset of the lengths of closed geodesics on $\mathbf{S}_g$, in the sense of shortest closed paths in their homotopy class. Then we expect
\begin{align*}
    (80 g)^{\frac{1}{4}}\,\Upxi_g \xrightarrow[g\to\infty]{\mathrm{(d)}} \mathcal{P},
\end{align*}
where $\mathcal{P}$ is the same Poisson point process on $\mathbb{R}_{+}$ with intensity $\frac{\cosh{t}-1}{t}\mathrm{d}t$ as above. In the convergence, the random multiset $\Upxi_g$ is regarded as a random point process on $[0,\infty)$.
The curious constant $(80)^{1/4}$ arises from the combination of $\alpha_2$ from Corollary~\ref{main_corollary} with the scaling constant $c_{\mathrm{WP}} = 2 \pi /\sqrt{3 j_0}$ of  \cite[Theorem~2]{budd2025random} or \cite[Corollary~6]{budd2025tree}, which satisfy $\alpha_2\,c_{\mathrm{WP}} = (80)^{1/4}$.
\end{conjecture}

\paragraph{Structure of the paper.}
In Section~\ref{preliminaries}, we give a recap of Weil-Petersson volumes and introduce notions related to tightness. In Section~\ref{Integration_formula}, we prove an integration formula that is adapted to \emph{tight multi-curves}. Then, in Section~\ref{estimates_tight_volumes}, we give asymptotics for the volumes $T_{g,n}((\mu_c-\mu_g)^{-\frac{1}{4}}\mathbf{L},\mu_g)$ where $g \to \infty$ relying intensively on \cite[Theorem~$4$]{budd2023topological}. Note that this part is interesting in itself since the method used also works for $g$ fixed. Finally, Section~\ref{proof_main_theorem} is dedicated to proving Theorem~\ref{main_theorem}, Corollary~\ref{main_corollary} and Corollary~\ref{systole}.

\paragraph{Acknowledgements.} 
This work is supported by the VIDI programme with project number VI.Vidi.193.048, which is financed by the Dutch Research Council (NWO). We are also grateful to the ENS Paris-Saclay for Lions's scholarship, which made this collaboration possible. We would also like to thank Bart Zonneveld, Joe Thomas, Will Hide, Baptiste Louf and Mingkun Liu for interesting and useful discussions about this work.

\clearpage

\tableofcontents

\section{Preliminaries}\label{preliminaries}
\subsection{Random hyperbolic surfaces}
\paragraph{Hyperbolic surfaces and Teichmüller space.} An oriented, compact, connected surface $X$ is said to be a \emph{hyperbolic surface} if it is equipped with a Riemmanian metric of constant curvature $-1$. For $g,n \ge 0$ such that $2g+n \ge 3$ we fix a compact, connected, oriented surface $\Sigma_{g,n}$ of genus $g$ with $n$ distinguished boundaries $\partial_1 \Sigma_{g,n},\cdots,\partial_n \Sigma_{g,n}$. We say that a pair $(X,\phi)$ is a marking if $X$ is an hyperbolic surface and $\phi : \Sigma_{g,n} \rightarrow X$ is a diffeomorphism. Two markings $(X,\phi)$ and $(Y,\psi)$ are said to be equivalent and we write $(X,\phi) \sim (Y,\psi)$ if there exists an isometry $h : X \to Y$ such that $\psi^{-1} \circ h \circ \phi : \Sigma_{g,n} \to \Sigma_{g,n}$ is isotopic to $\operatorname{id}_{\Sigma_{g,n}}$. For $\mathbf{L} = (L_1,\cdots,L_n) \in \mathbb{R}_{+}$, the \emph{Teichmüller space} $\mathcal{T}_{g,n}(\mathbf{L})$ is defined as
\begin{align*}
    \mathcal{T}_{g,n}(\mathbf{L}) := \left \{
\begin{array}{rcl}
 &(X,\phi) \text{ is a marking and }\\
(X,\phi) :& X \text{ has geodesic boundaries }\\ 
&\text{ of specified lengths }\mathbf{L}
\end{array}
\right\} \backslash \sim.
\end{align*}
In words, the marking is a way to equip our surface $\Sigma_{g,n}$ with a hyperbolic structure. However, one can see that the same hyperbolic surface admits multiple markings.
\begin{figure}[H]
    \centering
    \includegraphics[scale = 0.3]{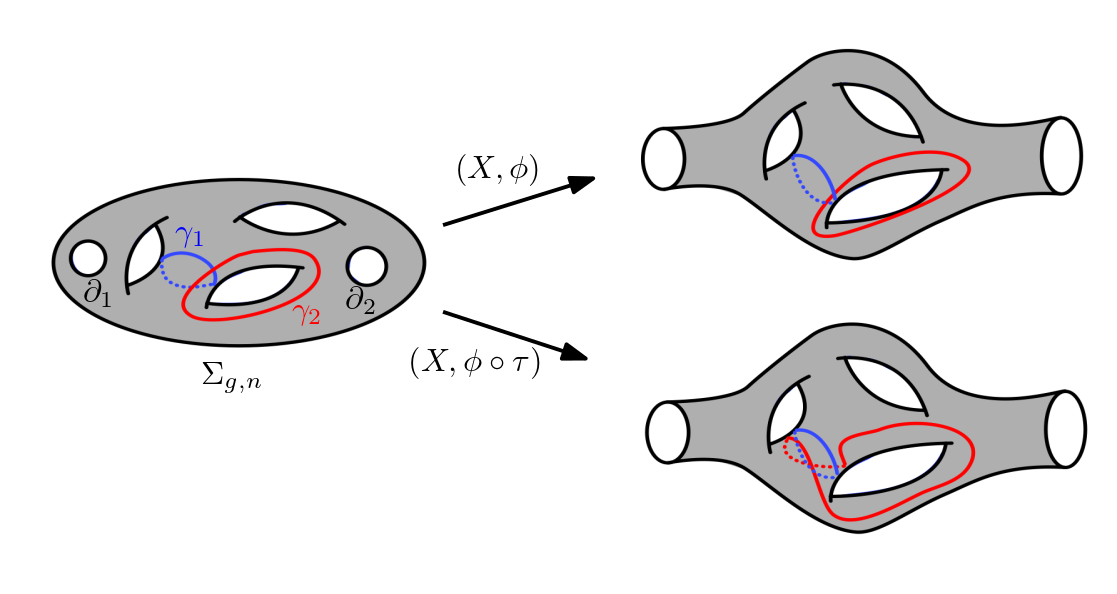}
    \caption{On the left the surface $\Sigma_{g,n}$. On the right, two examples of markings which give the same hyperbolic surface $X$, however this gives two different points in $\mathcal{T}_{g,n}(\mathbf{L})$. Indeed to obtain the bottom marking we have first applied a full twist $\tau$ along the geodesic $\gamma_1$ in $\Sigma_{g,n}$.}
    \label{marking}
\end{figure}

In the following we want to work with hyperbolic surfaces seen up to isometry, thus we want to forget the marking.
\paragraph{The moduli space of hyperbolic surface and the Weil-Petersson measure.}
We define the the \emph{mapping class group} $\text{MCG}_{g,n}$ as follows
\begin{align}\label{mapping_class_group_definition}
    \text{MCG}_{g,n} = \text{Diff}^{+}(\Sigma_{g,n}) \backslash \text{Diff}^{+}_{0}(\Sigma_{g,n}),
\end{align}
where
\begin{align*}
    \text{Diff}^{+}(\Sigma_{g,n}) := \left \{
\begin{array}{rcl}
 & \phi \text{ is an orientation }\\
\phi : \Sigma_{g,n} \to \Sigma_{g,n} :& \text{preserving  diffeomorphism}\\ 
&\text{that preserves the boundaries}
\end{array}
\right\}\\
\end{align*}
and
\begin{align*}
    \text{Diff}^{+}_{0}(\Sigma_{g,n}) = \bigg\{\phi \in \text{Diff}^{+}(\Sigma_{g,n}) : \phi \text{ is homotopic to }id_{\Sigma_{g,n}}\}.
\end{align*}
The space we will work with in the rest of the paper is called the \emph{moduli space} $\mathcal{M}_{g,n}(\mathbf{L})$ and is defined as
\begin{align}\label{moduli_space_definition}
    \mathcal{M}_{g,n}(\mathbf{L}) = \mathcal{T}_{g,n}(\mathbf{L}) \backslash \text{MCG}_{g,n}.
\end{align}
The moduli space can be seen as the space of hyperbolic surfaces seen up to isometry with $n$ distinguished geodesic boundaries of specified length $\mathbf{L}$.\\
The moduli space $\mathcal{M}_{g,n}(\mathbf{L})$ is much smaller than $\mathcal{T}_{g,n}(\mathbf{L})$ since different markings may correspond to the same hyperbolic surface in $\mathcal{M}_{g,n}(\mathbf{L})$ (see Figure~\ref{marking}). 
We consider the Weil-Petersson measure $\mu_{g,n}^{\mathrm{WP}}$ on $\mathcal{M}_{g,n}(\mathbf{L})$, which is well known to have finite total volume $\mu_{g,n}^{\mathrm{WP}}(\mathcal{M}_{g,n}(\mathbf{L}))$.
The Weil-Petersson volumes are defined as $V_{g,n}(\mathbf{L}) = \mu_{g,n}^{\mathrm{WP}}(\mathcal{M}_{g,n}(\mathbf{L}))$ for $2g+n > 3$, $V_{0,3}(\mathbf{L}) = 1$ and $V_{1,1}(\mathbf{L}) = \frac{1}{2}\mu_{g,n}^{\mathrm{WP}}(\mathcal{M}_{1,1}(\mathbf{L}))$.
The conventional factor of $\frac{1}{2}$ for the volume $V_{1,1}(\mathbf{L})$ reflects the non-trivial automorphisms of surfaces in $\mathcal{M}_{1,1}(\mathbf{L})$, leading to simpler formulas in the following.
For $2g+n > 3$, we write $\Pf_{g,n}^{\mathrm{WP}} = V_{g,n}^{-1}\mu_{g,n}^{\mathrm{WP}}$ which defines a probability measure on $\mathcal{M}_{g,n}(\mathbf{L})$. 

\paragraph{Closed curves, geodesics and multi-curves.}
For $S$ a surface, a \emph{closed curve} is a continuous map $c : [0,1] \to S$ such that $c(0) = c(1)$. We say that $c$ is \emph{simple} if it does not self-intersect. We say that two closed curves $c_1$ and $c_2$ are \emph{freely-homotopic} if there exists a continous map $h : [0,1]^{2} \to S$ such that $h|_{\{0\}\times [0,1]} = c_1$ and  $h|_{\{1\}\times [0,1]} = c_2$ and $h(t,0)=h(t,1)$ for $t\in[0,1]$. This defines an equivalence relation. We will often consider curves up to free-homotopy and we use the notation $[c]_S$ to denote the free-homotopy class of a curve $c$ in $S$.

For $X$ a hyperbolic surface and $c$ a closed curve on $X$ that is non-contractible, there exists a unique geodesic in its free-homotopy class (see \cite[Theorem $1.6.6$]{buser1992geometry}), thus we usually identify $[c]$ with the unique geodesic in the equivalence class and we define $\ell_{X}([c])$ its length.

A \emph{multi-curve} is a set of disjoint non-contractible simple closed curves $(\gamma_1,\cdots,\gamma_k)$ such that for any $1\le i < j \le k$, the curve $\gamma_i$ is freely homotopic to neither $\gamma_j$ or $\gamma^{-1}_j$. On a hyperbolic surface $X$, any multi-curve $\Gamma = (\gamma_1,\cdots,\gamma_k)$ has a representative in its free-homotopy class which is made of geodesics (see \cite[Theorem $1.6.6-7$]{buser1992geometry}). In words, this means that it is possible to tighten the curves $\gamma_i$ while keeping them simple and disjoint. We can then define the \emph{total length} and \emph{length vector} of a multi-curve as
\begin{align*}
    \ell_{X}(\Gamma) = \somme{i=1}{k}{\ell_{X}(\gamma_i)}\hspace{0.5cm}\text{ and }\hspace{0.5cm}   \vec{\ell}_{X}(\Gamma) = (\ell_X(\gamma_1),\cdots,\ell_X(\gamma_k)).
\end{align*}
\subsection{Tight closed geodesics and Boltzmann random hyperbolic surfaces}
\paragraph{Tight geodesics.}
Fix $n,g,p \geq 0$ such that $2g + n+p > 3$. We define $\Sigma_{g,n,p}$ to be the surface obtained by capping off the last $p$ boundaries of $\Sigma_{g,n+p}$ with disks. If $\gamma$ is a closed curve in $\Sigma_{g,n+p}$, we may consider its \emph{extended} free homotopy class $[\gamma]_{\Sigma_{g,n,p}}$ by viewing $\gamma$ as a curve on $\Sigma_{g,n,p}$. Note that when $p>0$, this extended free homotopy class naturally includes infinitely many free homotopy classes of $\Sigma_{g,n+p}$.
\begin{definition1}\label{def_tightness}
 For $(X,\phi) \in \mathcal{T}_{g,n+p}(\mathbf{L})$ a marking, a closed geodesic $\gamma$ on $X$ is said to be a \emph{tight closed geodesic} if $\gamma$ is not contractible in $\Sigma_{g,n,p}$ and all other closed geodesics $\gamma'$ on $X$ belonging to the same extended free homotopy class,  $\gamma'\in [\gamma]_{\Sigma_{g,n,p}}$,  have length $\ell_X(\gamma') > \ell_X(\gamma)$.
\end{definition1}
Note that this definition depends on $p$ and not only on $n+p$. Indeed, for $p=0$ all geodesics are tight while for $p > 0$, since we 'cap off' $p$ boundaries, there are fewer tight geodesics. In the introduction, we considered this definition in the case where $n = 0$, i.e.\ curves are allowed to pass over all boundaries (or cusps).

 One can check that the tightness is invariant under the action of MCG$_{g,n+p}$. Thus, it is correctly defined on $\mathcal{M}_{g,n+p}(\mathbf{L})$. Note that $\partial_1 X,\cdots,\partial_n X$ are not necessarily tight closed geodesics. However, we will later be interested in that case. See \cite{budd2023topological} for an introduction of hyperbolic surfaces with tight boundaries. The reader may also be interested in \cite{AHL_2022__5__1035_0,bouttier2022quasipolynomials} where the authors introduce the analogous notion for planar maps.
\begin{figure}[H]
    \centering
    \includegraphics[scale = 1.0]{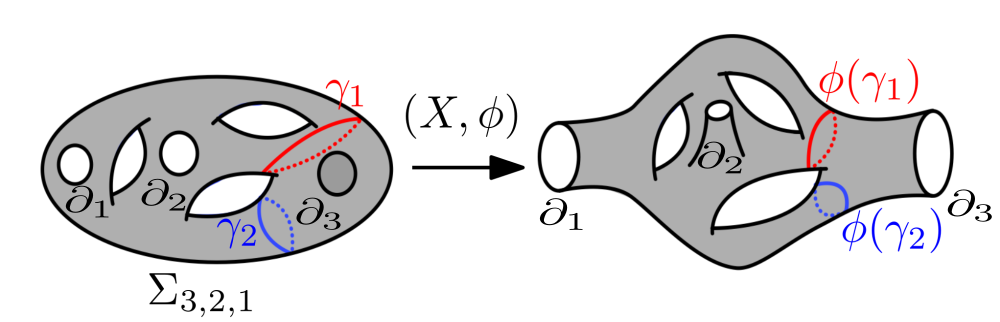}
    \caption{We take $g=3, n=2, p=1$ and we consider a marking $(X,\phi)$. The geodesic $\phi(\gamma_1)$ is not tight since $\phi(\gamma_2)$ has a smaller length and $\gamma_1$ and $\gamma_2$ are homotopic in $\Sigma_{3,2,1}$.}
    \label{tight_geodesic}
\end{figure}
In words, a tight closed geodesic is the shortest curve in the extended homotopy class where curves are allowed to pass over the boundaries $\partial_{n+1},\cdots,\partial_{n+p}$ (see Figure~\ref{tight_geodesic}). Note that if a simple closed geodesic separates a surface of genus $0$ containing none of the boundaries $\partial_1,\cdots,\partial_n$, then this geodesic cannot be tight since it is homotopic to a point in the extended homotopy class (see Figure~\ref{go_to_tight}). In particular, if $n = 0$, then the simple tight closed geodesics are genus reducing or separating. Reciprocally, for $n =0$, to any simple closed geodesic $\gamma$ that is genus reducing or separating, there exists a simple closed geodesic of minimal length in its extended homotopy class (see \cite[Theorem $1.6.11$]{buser1992geometry}). Furthermore, for $\mu^{\mathrm{WP}}_{g,n}$-almost every $X$, all simple closed geodesics have distinct length, thus the representative is unique and tight (see Figure~\ref{go_to_tight}). Moreover, in the latter case, if we cut the surface along the tight representative, we obtain a surface with tight geodesic boundaries.
In fact, we will see below (Lemma~\ref{lem:tightcomponents}) that the reverse holds as well: if one glues surfaces along tight geodesics boundaries of matching length, the gluing seams are tight geodesics.
\begin{figure}[H]
    \centering
    \includegraphics[scale=0.25]{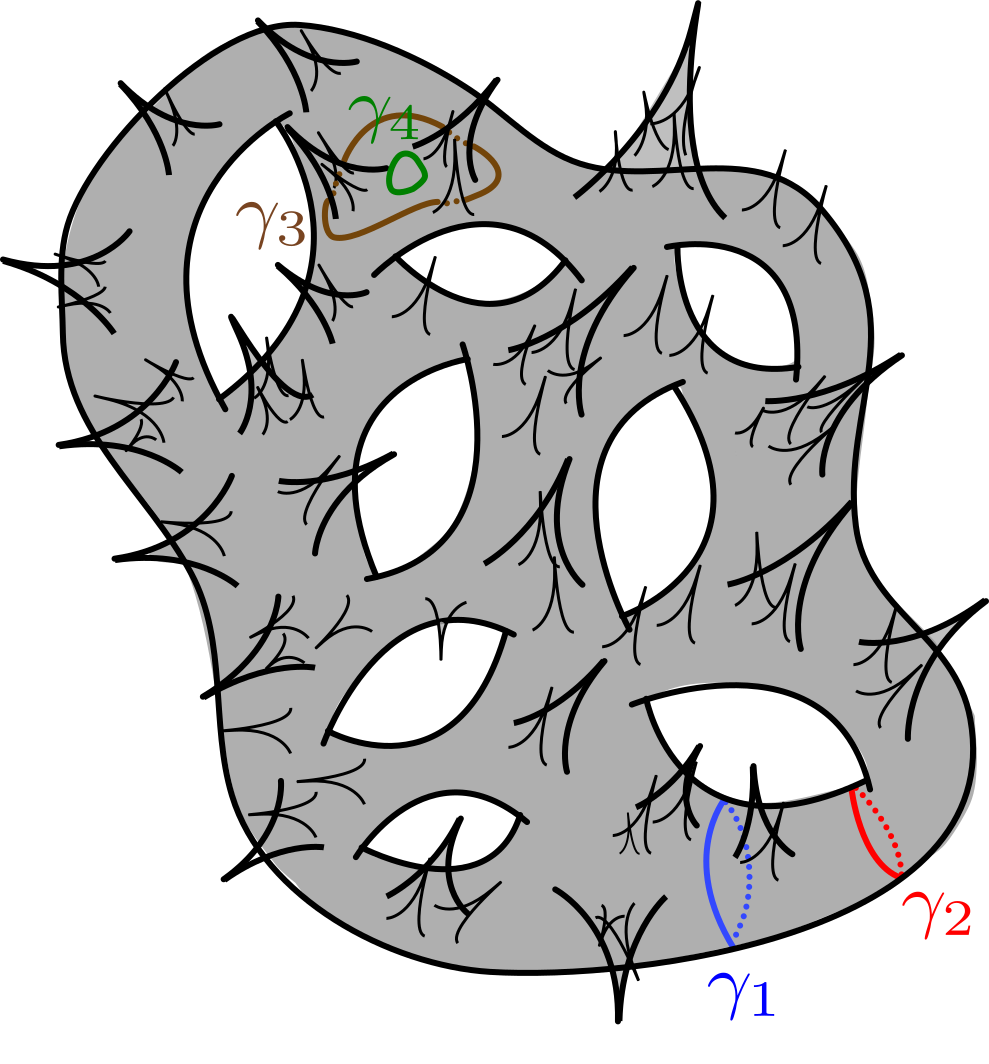}
    \caption{In this example, with $n=0$, the curves $\gamma_3$ and $\gamma_4$ are part of the same extended homotopy class, thus none of them is tight. The curve $\gamma_2$ is the tight representative of $\gamma_1$.}
    \label{go_to_tight}
\end{figure}

This motivates of the definition the moduli space of hyperbolic surfaces with $n$ tight boundaries by setting
\begin{align*}
    \mathcal{M}^{\mathrm{tight}}_{g,n,p}(\mathbf{L}) = \{X \in \mathcal{M}_{g,n+p}(\mathbf{L})  \text{ : }\partial_1,\cdots, \partial_n \text{ are tight}\} \subset \mathcal{M}_{g,n+p}(\mathbf{L}).
\end{align*}
This defines an open subset of $\mathcal{M}_{g,n+p}(\mathbf{L})$, so we can consider the Weil-Petersson measure on this set. We define $T_{g,n,p}(\mathbf{L}) = \mu_{g,n+p}^{\mathrm{WP}}(\mathcal{M}^{\mathrm{tight}}_{g,n,p}(\mathbf{L})) \le V_{g,n+p}$. Note that $T_{g,n,0}(\mathbf{L}) = T_{g,0,n}(\mathbf{L}) = V_{g,n}(\mathbf{L})$.

In the following, we will think of the tight boundaries as the real boundaries and the number of non-tight boundaries $p$ will be made random.
\paragraph{Boltzmann hyperbolic surfaces.}\label{Boltzmann hyperbolic surfaces}
For $\mu \ge 0$, $g,n \ge 0$ and $\mathbf{L} \in \mathbb{R}_{+}^{n}$, we define
\begin{align}\label{boltzmann_volume}
    &F_{g,n}(\mathbf{L},\mu) = \somme{p=0}{\infty}{\mu^{p}\frac{V_{g,n+p}(\mathbf{L},\mathbf{0}^{p})}{p!}}, \hspace{1cm} T_{g,n}(\mathbf{L},\mu) = \somme{p=0}{\infty}{\mu^{p}\frac{T_{g,n,p}(\mathbf{L},\mathbf{0}^{p})}{p!}}.
\end{align}
If $F_{g,n}(\mathbf{L},\mu)$ is finite we define the \emph{$\mu$-Boltzmann probability measure} $\Pf^{\mathrm{WP}}_{g,\mu,n,\mathbf{L}}$ on $\displaystyle \bigsqcup_{p=0}^{\infty}\mathcal{M}_{g,n+p}(\mathbf{L})$ by first sampling $p$ with probability $\displaystyle F_{g,n}(\mathbf{L},\mu)^{-1}\mu^{p}\frac{V_{g,n+p}(\mathbf{L},\mathbf{0}^{p})}{p!}$ and then choosing $X$ under $\Pf^{\mathrm{WP}}_{g,n,(\mathbf{L},0^{p})}$. We call $\mathcal{N}_{g,\mu,n,\mathbf{L}}$ the number of cusps of the surface $X$ chosen under $\Pf^{\mathrm{WP}}_{g,\mu,n,\mathbf{L}}$.\\

\noindent If $T_{g,n}(\mathbf{L},\mu)$ is finite, we define the \emph{tight $\mu$-Boltzmann probability measure} $\Pf^{\mathrm{tight}}_{g,\mu,n,\mathbf{L}}$ on $\displaystyle \bigsqcup_{p=0}^{+\infty}\mathcal{M}^{\mathrm{tight}}_{g,n+p}(\mathbf{L})$ by first sampling $p$ with probability $\displaystyle T_{g,n}(\mathbf{L},\mu)^{-1}\mu^{p}\frac{T_{g,n,p}(\mathbf{L},\mathbf{0}^{p})}{p!}$ and then choosing $X$ under the Weil-Petersson probability measure on $\displaystyle \mathcal{M}^{\mathrm{tight}}_{g,n+p}(\mathbf{L},\mathbf{0}^{p})$. We call $\mathcal{N}^{\mathrm{tight}}_{g,\mu,n,\mathbf{L}}$ the number of cusps of the surface $X$ chosen under $\Pf^{\mathrm{tight}}_{g,\mu,n,\mathbf{L}}$.\\

\noindent Observe that in the case $n = 0$ or $L = \mathbf{0}^{n}$, we have $\mathcal{M}^{\mathrm{tight}}_{g,n,p}(\mathbf{L},\mathbf{0}^{p}) = \mathcal{M}_{g,n+p}(\mathbf{L},\mathbf{0}^{p})$. It follows that in these cases we have $\Pf^{\mathrm{tight}}_{g,\mu,n,\mathbf{L}} = \Pf^{\mathrm{WP}}_{g,\mu,n,\mathbf{L}}$. In the case of $n = 0$, we shorten the notation by omitting $n$ and $\mathbf{L}$. When $n \ge 1$ but $\mathbf{L} = \mathbf{0}^n$ we omit $\mathbf{L}$ in the notation.

\paragraph{Recursion formula for tight volumes.}
 We introduce
 \begin{align}\label{def_Z}
     Z(r,\mu) = \frac{\sqrt{r}}{\sqrt{2}\pi}J_1(2\pi \sqrt{2r}) - \mu. 
 \end{align}
We also define $R(\mu)$ the unique formal power series in $\mu$ such that $Z(R(\mu),\mu) = 0$ and $R(0) = 0$. Then, we define
\begin{align}\label{moments}
    M_k(\mu) = \frac{\partial^{k+1}Z}{\partial r^{k+1}}\bigg(R(\mu),\mu\bigg) = \bigg(\frac{-\sqrt{2}\pi}{\sqrt{R(\mu)}}\bigg)^{k}J_k(2\pi \sqrt{2R(\mu)}),
\end{align}
where $J_k$ denotes the $k^{th}$ Bessel function of the first kind and the second equality follows from \cite[Equation $(23)$]{budd2023topological}. The moments $M_k$ will play an important role in the following. Let us define $\mu_c = \inf\{\mu \ge 0 \text{ : }M_0(\mu)=0\}$ . Note that $M_0(\mu) \underset{\mu \to \mu_c}{\rightarrow} 0$. We can now state the main result to compute the generating functions $T_{g,n}(\mathbf{L},\mu)$.
\begin{theorem}[\cite{budd2023topological}, Theorem $4$]\label{polynomial_recursion}
    For all $g,n \ge 0$ such that $n \ge 3$ for $g =0$ and $n\ge 1$ for $g=1$, and any $\mu \le \mu_c$, we have
    \begin{align*}
        T_{g,n}(\mathbf{L},\mu) = \frac{1}{M_0(\mu)^{2g-2+n}}\mathcal{P}_{g,n}\bigg(\mathbf{L},\frac{M_1(\mu)}{M_0(\mu)},\cdots,\frac{M_{3g-3+n}(\mu)}{M_0(\mu)}\bigg),
    \end{align*}
    where $\mathcal{P}_{g,n}(\mathbf{L},\mathbf{m})$ is a rational polynomial in $L_1^2,\cdots,L_n^{2},m_1,\cdots,m_{3g-3+n}$. This polynomial is symmetric and of degree $3g-3+n$ in $L_1^2,\cdots,L_n^{2}$ while $\mathcal{P}_{g,n}(\sqrt{\sigma}\mathbf{L},\sigma m_1,\cdots,\sigma^{3g-3+n}m_{3g-3+n})$ is homogeneous of degree $3g-3+n$ in $\sigma$. For $g \ge 0$, $n \ge 1$ such that $2g-3+n > 0$, the polynomial $\mathcal{P}_{g,n}(\mathbf{L},\mathbf{m})$ can be obtained from $\mathcal{P}_{g,n-1}(\mathbf{L},\mathbf{m})$ via the recursion relation
    \begin{align}\label{recursion_relation}
        &\mathcal{P}_{g,n}(\mathbf{L},\mathbf{m}) = \somme{p=1}{3g-4+n}{\bigg( m_{p+1} - \frac{L_1^{2p+2}}{2^{p+1}(p+1)!} - m_1m_p + \frac{1}{2}L_1^{2}m_p\bigg)\frac{\partial \mathcal{P}_{g,n-1}}{\partial m_p}(\mathbf{L}_{\widehat{\{1\}}},\mathbf{m})}\\
        &\qquad +(2g-3+n)\bigg(-m_1+\frac{1}{2}L_1^2\bigg)\mathcal{P}_{g,n-1}(\mathbf{L}_{\widehat{\{1\}}},\mathbf{m}) + \somme{i=2}{n}{\inte{0}{L_i}{x\mathcal{P}_{g,n-1}(x,\mathbf{L}_{\widehat{\{1,i\}}},\mathbf{m})}{x}}.\nonumber
    \end{align}
Furthermore, we have:
\begin{align}\label{polynomial_base_case}
    &\mathcal{P}_{0,3}(\mathbf{L},\mathbf{m}) = 1,\\
    &\mathcal{P}_{1,1}(\mathbf{L},\mathbf{m}) = \frac{1}{24}(-m_1+\frac{1}{2}L_1^{2}),
\end{align}
and $\mathcal{P}_{g,0}(\mathbf{m})$ is given by:
\begin{align}\label{Pg0_expression}
    \mathcal{P}_{g,0}(\mathbf{m}) = \somme{\substack{d_2,d_3\cdots \ge 0\\ \somme{k=2}{3g-2}{(k-1)d_k} = 3g-3}}{}{\langle \tau_{2}^{d_2}\cdots\tau_{3g-2}^{d_{3g-2}}\rangle_g \prod_{k=2}^{3g-2}\frac{(-m_{k-1})^{d_k}}{d_k!} },
\end{align}
where $\langle \tau_{2}^{d_2}\cdots\tau_{3g-2}^{d_{3g-2}}\rangle_g$ are the $\psi$-class intersection numbers on the moduli space $\mathcal{M}_{g,n}$ with $n = \somme{k}{}{d_k} \le 3g-3$ marked points.
\end{theorem}

\section{The integration formula for random tight hyperbolic surfaces}\label{Integration_formula}

This section is dedicated to proving an integration formula similar to Mirzakhani's one but for the setting of random tight hyperbolic surfaces.

For the purpose of this section it is useful to fix a surface $\Sigma_{g,n}$ for every $g \geq 0, n\geq 0$ satisfying $2g+n\geq 3$ in such a way that $\Sigma_{g,n + p} \subset \Sigma_{g,n}$ for $p>0$ (see Figure~\ref{map_lambda}).
In other words, $\Sigma_{g,n}$ is the result of capping the last $p$ boundaries of $\Sigma_{g,n + p}$.
For a multicurve $\Gamma'$ on $\Sigma_{g,n+p}$ we denote its free homotopy class by $[\Gamma']_{g,n+p}$, and we observe that there is a natural inclusion $[\Gamma']_{g,n+p} \subset [\Gamma']_{g,n}$ and we call $[\Gamma']_{g,n}$ the extended free homotopy class of $\Gamma'$.
For $\Gamma$ a multicurve on $\Sigma_{g,n}$ we denote by  
$\mathcal{C}^{\Gamma}_{g,n+p} = \{ [\Gamma']_{g,n+p} : [\Gamma']_{g,n+p}\subset[\Gamma]_{g,n}\}$ the collection of free homotopy classes of multicurves on $\Sigma_{g,n+p}$ included in that of $\Gamma$ in $\Sigma_{g,n}$.
We begin by the geometric observation that multicurves can be tightened without introducing new intersections.

\begin{proposition}\label{tighten_multicurve}
    Let $X \in \mathcal{M}_{g,n+p}(\mathbf{L},\mathbf{K})$ and $\Gamma=(\gamma_1,\ldots,\gamma_k)$ be a multicurve on $\Sigma_{g,n+p}$ whose curves are simple and pairwise disjoint and not freely homotopic in $\Sigma_{g,n}$ to each other, to a boundary or to a point. For each $i=1,\ldots, k$, there exists a closed geodesic $\gamma_i' \in [\gamma_i]_{g,n}$ of minimal length. Then the multicurve $\Gamma' = \{\gamma_1',\ldots,\gamma_k'\}$ consists of simple and pairwise disjoint curves.
    Furthermore, $\Gamma'$ is tight or $X$ contains two simple closed geodesics of equal length that are not boundaries.
\end{proposition}
\begin{proof}
    We note that there are only finitely many closed geodesics in $X$ of length smaller than $\ell_X(\gamma_i)$ so a closed geodesic $\gamma_i'$ of minimal length in $[\gamma_i]_{g,n}$ necessarily exists.
    The claim that the length-minimizing curves do not have any intersections holds more generally for Riemannian 2-manifolds so instead of giving a self-contained proof, we refer to \cite[Theorem~2.1 \& Corollary~3.4]{freedman1982}.
    To fulfill their conditions, we smoothly cap off the last $p$ boundaries of $X$ with disks, making sure not to introduce shortcuts through the interior of the disk between pairs of points on each boundary.
    In case we are dealing with a cusp, we remove a sufficiently small horocyclic neighbourhood which we then cap off.
    The geodesics $\gamma_i'$ will be length-minimizing in the resulting Riemmanian manifold as well.
    Thus we may apply their results, taking into account that the remaining boundaries are convex (see \cite[\S 4]{freedman1982}), to deduce that the length-minimizing geodesics are simple and pairwise disjoint.

    If the length-minimizing geodesic $\gamma_i'$ is unique, then it is a tight geodesic by definition.
    Otherwise, by \cite[Corollary~3.4]{freedman1982}, $X$ supports a pair of disjoint simple closed geodesics of equal length in $[\gamma_i]_{g,n}$.
    This verifies the final statement.
\end{proof}

\begin{figure}[H]
    \centering
    \includegraphics[scale = 0.3]{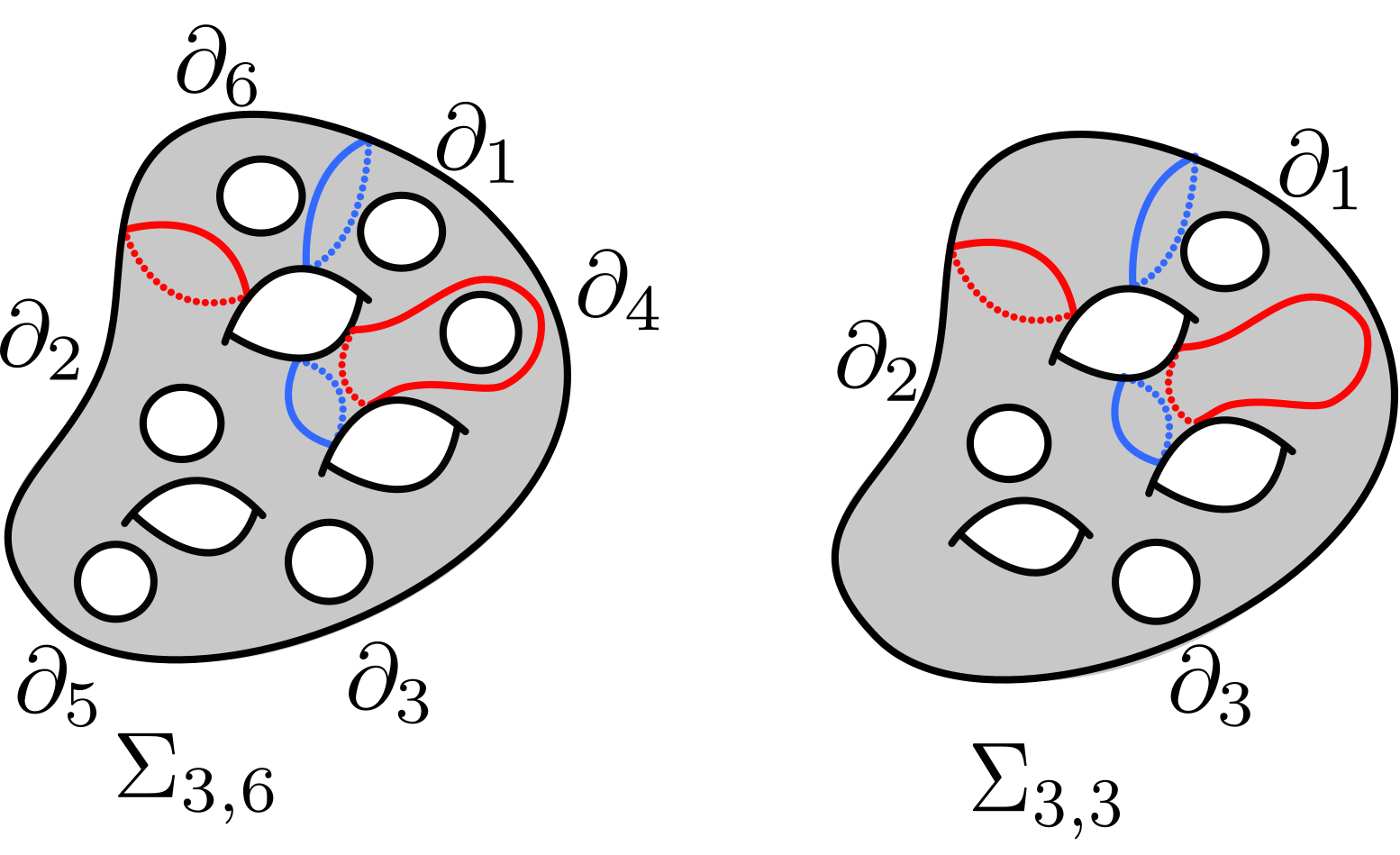}
    \caption{On this example we take $n=3$ and $p=3$. The surface $\Sigma_{3,6}$ is a subsurface of $\Sigma_{3,3}$. Call $\Gamma$ the multicurve consisting of the blue curves and $\Gamma^{'}$ the multicurve consisting of the red curves. On this example we have $[\Gamma]_{3,6} \ne [\Gamma^{'}]_{3,6}$ but $[\Gamma^{'}]_{3,6} \in \mathcal{C}^{\Gamma}_{3,6}$.}
    \label{map_lambda}
\end{figure}

Suppose now $\Gamma = (\gamma_1,\ldots,\gamma_k)$ is a multi-curve on $\Sigma_{g,n}$ such that none of the curves $\gamma_1,\ldots,\gamma_k$ is freely homotopic to a boundary.
As before, we assume that cutting along $\Gamma$ decomposes $\Sigma_{g,n}$ into $q\geq 1$ components, $\Sigma_{g,n} \setminus \Gamma = \bigsqcup_{i=1}^q \Sigma_{g_i,n_i+k_i}$, where the $i$th component is a genus-$g_i$ surface with $n_i$ boundaries corresponding to boundaries of $\Sigma_{g,n}$ and $k_i$ boundaries corresponding to sides of curves in $\Gamma$.
If $\mathbf{L}\in \mathbb{R}_+^n$ and $\mathbf{x}\in \mathbb{R}_+^{k}$ are vectors indexed by respectively the boundaries of $\Sigma_{g,n}$ and the curves of $\Gamma$, we denote by $\mathbf{L}^{(i)} \in \mathbb{R}_+^{n_i}$ and $\mathbf{x}^{(i)} \in \mathbb{R}_+^{k_i}$ the associated vectors indexed by the boundaries of $\Sigma_{g_i,n_i+k_i}$.

For $p > 0$, $\mathbf{L}\in \mathbb{R}_+^n$, and $\mathbf{K}\in \mathbb{R}_+^p$, given a hyperbolic surface $X\in \mathcal{T}_{g,n+p}(\mathbf{L},\mathbf{K})$ we can define the \emph{tight length vector} of the multi-curve $\Gamma$ on $\Sigma_{g,n}$ to be
\begin{align*}
    \vec{\ell}_X^{\;\mathrm{tight}}(\Gamma) = (\ell_X^{\mathrm{tight}}(\gamma_1),\ldots,\ell_X^{\mathrm{tight}}(\gamma_k)),
\end{align*}
where $\ell_X^{\mathrm{tight}}(\gamma_i)$, $i=1,\ldots,k$, is the minimum of lengths $\ell_X(\gamma)$ of all closed geodesics $\gamma$ in $X$ that belong to the free homotopy class of $\gamma_i$ in $\Sigma_{g,n}$.


Given a function $f : \mathbb{R}_+^k\to \mathbb{R}_+$, let 
\begin{align*}
    F^\Gamma_{\mathrm{tight}}(X) = \sum_{\Gamma' \in \text{Orb}_{\mathrm{MCG}_{g,n}}(\Gamma)}  f\left(\vec{\ell}_X^{\;\mathrm{tight}}(\Gamma')\right).
\end{align*}

\begin{proposition} \label{prop:integrationformula}
Mirzakhani's integration formula generalizes to an integration formula on the moduli space of tight hyperbolic surfaces given by 
    \begin{align*}
        \int_{\mathcal{M}_{g,n,p}^{\mathrm{tight}}(\mathbf{L},\mathbf{K})} F^\Gamma_{\mathrm{tight}}(X) \mathrm{d}\mu_{g,n+p}^{\mathrm{WP}}(X) = C_{\Gamma}\int_{\mathbb{R}_+^k} x_1\cdots x_k f(\mathbf{x}) T_{g,n,p}(\mathbf{L},\mathbf{K},\Gamma,\mathbf{x}) \mathrm{d}x_1\cdots \mathrm{d}x_k,
    \end{align*}
    where $T_{g,n,p}(\mathbf{L},\mathbf{K},\Gamma,\mathbf{x})$ is defined with reference to the decomposition $\Sigma_{g,n} \setminus \Gamma = \bigsqcup_{i=1}^q \Sigma_{g_i,n_i+k_i}$ as 
    \begin{align*}
        T_{g,n,p}(\mathbf{L},\mathbf{K},\Gamma,\mathbf{x}) = \sum_{I_1\sqcup \cdots\sqcup I_q = \{1,\ldots,p\}} \prod_{i=1}^q T_{g_i,n_i+k_i,|I_i|}(\mathbf{L}^{(i)},\mathbf{x}^{(i)},\mathbf{K}_{I_i}).
    \end{align*}
    The constant $C_{\Gamma}$ only depends on $\Gamma$. We have $C_{\Gamma} \le 1$ and if $\Sigma_{g,n}\backslash \Gamma$ is connected, then $C_{\Gamma} = 2^{-k}$.
\end{proposition}

The proof of this formula relies on a specific symplectomorphism introduced in the work \cite[Sec.~8]{Mirzakhani_polynomial_volumes} of Mirzakhani and on a specialisation to the tight case discussed in \cite[Sec.~2.3]{budd2023topological}, which we recap here.
For the moment we ignore the tightness condition and focus on the ordinary moduli space $\mathcal{M}_{g,n}(\mathbf{L})$.
We consider the stabilizer subgroup 
\begin{align*}
    \mathrm{Stab}_{g,n}(\Gamma) = \{ h\in \text{MCG}_{g,n} : h \cdot \gamma_i = \gamma_i, \,i=1,\ldots,k\} \subset \text{MCG}_{g,n}
\end{align*}
of the multi-curve $\Gamma$, which gives rise to the moduli space 
\begin{align*}
    \mathcal{M}_{g,n}(\mathbf{L})^{\Gamma} = \mathcal{T}_{g,n}(\mathbf{L}) / \mathrm{Stab}_{g,n}(\Gamma).
\end{align*}
of hyperbolic surfaces carrying $\Gamma$ as a distinguished multi-curve.
Let $\mathcal{M}_{g,n}(\mathbf{L})^{\Gamma}(\mathbf{x})$ for $\mathbf{x}\in\mathbb{R}_+^k$ be the level set of the geodesic length function, i.e.\ the set of surfaces $X$ in which the geodesic representatives of the distinguished curves have prescribed lengths $\ell_X(\gamma_i) = x_i$ for $i=1,\ldots,k$.
This space admits a natural $k$-dimensional torus action via twisting along the distinguished curves, and the quotient space is denoted $\mathcal{M}_{g,n}(\mathbf{L})^{\Gamma*}(\mathbf{x})$.
It inherits a symplectic structure from the Weil--Petersson symplectic structure on $\mathcal{T}_{g,n}$.
If $\Sigma_{g,n}\setminus \Gamma$ decomposes as $\bigsqcup_{i=1}^q \Sigma_{g_i,n_i+k_i}$, then cutting $X$ along $\Gamma$ determines a canonical mapping 
\begin{align}
    \mathcal{M}_{g,n}(\mathbf{L})^{\Gamma*}(\mathbf{x})  \longrightarrow \prod_{i=1}^q \mathcal{M}_{g_i,n_i+k_i}(\mathbf{L}^{(i)},\mathbf{x}^{(i)}).\label{eq:cuttingsymplectomorphism}
\end{align}
According to \cite[Lemma~8.3]{Mirzakhani_polynomial_volumes} this is a symplectomorphism, so one may easily relate integration on one side to integration on the other.

Returning to the tight setting, we note that by extension of homeomorphisms on $\Sigma_{g,n+p}$ to $\Sigma_{g,n}$ one obtains a natural morphism $\lambda_{g,n,p} : \text{MCG}_{g,n+p} \to \text{MCG}_{g,n}$ of the corresponding mapping class groups.
We consider the subgroup $\lambda_{g,n,p}^{-1}(\mathrm{Stab}_{g,n}(\Gamma)) \subset \text{MCG}_{g,n+p}$ of mapping class group elements on $\Sigma_{g,n+p}$ that preserve the free homotopy class of $\Gamma$ in $\Sigma_{g,n}$.
We may then think of 
\begin{align*}
    \mathcal{M}_{g,n+p}(\mathbf{L},\mathbf{K})^{[\Gamma]_{g,n}} \coloneqq \mathcal{T}_{g,n+p}(\mathbf{L}) / \lambda^{-1}_{g,n,p}(\mathrm{Stab}_{g,n}(\Gamma))
\end{align*}
as the moduli space of hyperbolic surfaces with a distinguished extended free homotopy class of multi-curves $\Gamma$, as opposed to the previously introduced moduli space
\begin{align*}
    \mathcal{M}_{g,n+p}(\mathbf{L},\mathbf{K})^{[\Gamma']_{g,n+p}} = \mathcal{T}_{g,n+p}(\mathbf{L}) / \mathrm{Stab}_{g,n+p}(\Gamma')
\end{align*}
of surfaces with a distinguished multi-curve $\Gamma'$ on $\Sigma_{g,n+p}$.
Note that if $\Gamma'\in\mathcal{C}_{g,n+p}^\Gamma$, then $\mathrm{Stab}_{g,n+p}(\Gamma')$ is a subgroup of $\lambda^{-1}_{g,n,p}(\mathrm{Stab}_{g,n}(\Gamma))$ and we have a natural mapping
\begin{align}
\mathcal{M}_{g,n+p}(\mathbf{L},\mathbf{K})^{[\Gamma']_{g,n+p}} \to \mathcal{M}_{g,n+p}(\mathbf{L},\mathbf{K})^{[\Gamma]_{g,n}}\label{eq:forgethomotopy}
\end{align}
that only retains the extended homotopy class of $\Gamma'$.
\begin{figure}[H]
    \centering
    \includegraphics[scale = 0.35]{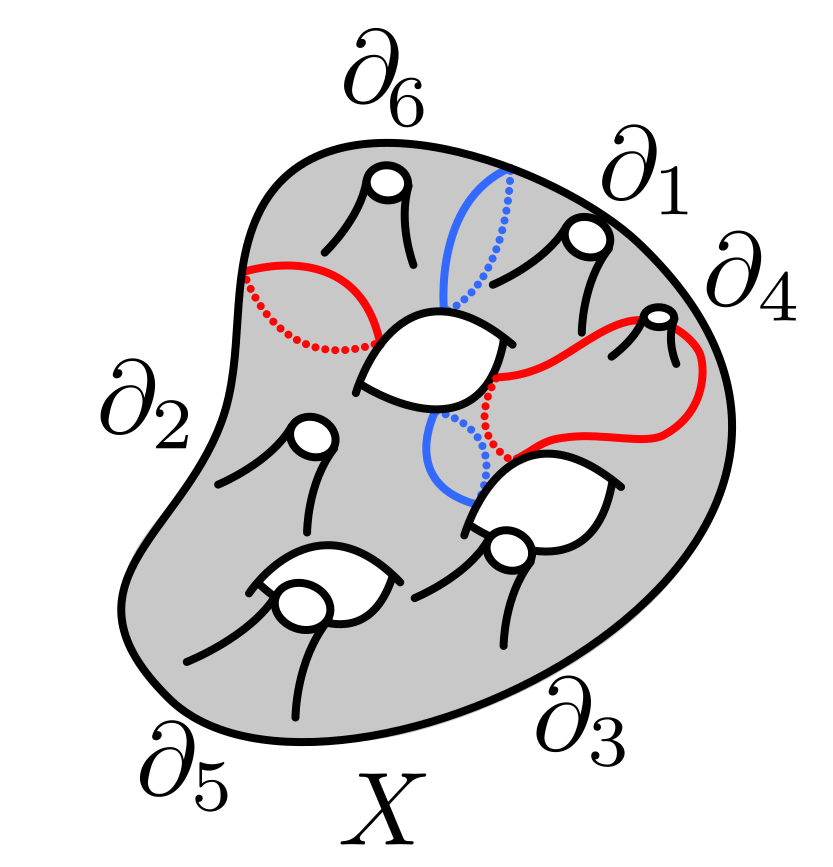}
    \caption{On this example we take $n=3$ and $p=3$. Call $\Gamma$ the multicurve consisting of the blue curves and $\Gamma^{'}$ the multicurve consisting of the red curves. We have $(X,\Gamma^{'}) \notin \mathcal{M}_{3,6}(\mathbf{L},\mathbf{K})^{[\Gamma]_{3,6}}$ but $(X,\Gamma^{'}) \in \mathcal{M}_{3,6}(\mathbf{L},\mathbf{K})^{[\Gamma]_{3,3}}$.}
    \label{fig:enter-label}
\end{figure}

In both cases one may restrict the hyperbolic surfaces to have $n$ tight boundaries, giving rise to
\begin{align*}
    \mathcal{M}^{\mathrm{tight}}_{g,n,p}(\mathbf{L},\mathbf{K})^{[\Gamma]_{g,n}} \subset \mathcal{M}_{g,n+p}(\mathbf{L},\mathbf{K})^{[\Gamma]_{g,n}},\qquad \mathcal{M}^{\mathrm{tight}}_{g,n,p}(\mathbf{L},\mathbf{K})^{[\Gamma']_{g,n+p}} \subset \mathcal{M}_{g,n+p}(\mathbf{L},\mathbf{K})^{[\Gamma']_{g,n+p}}.
\end{align*}
For a hyperbolic surface $X\in \mathcal{M}^{\mathrm{tight}}_{g,n,p}(\mathbf{L},\mathbf{K})^{[\Gamma']_{g,n+p}}$ we say that $\gamma_i'$ is tight if its geodesic representative is the unique curve of minimal length among all curves in its extended free homotopy class $[\gamma_i']_{g,n}$.
This allows us to further restrict the distinguished multicurve to be tight,
\begin{align*}
    \hat{\mathcal{M}}^{\mathrm{tight}}_{g,n,p}(\mathbf{L},\mathbf{K})^{[\Gamma']_{g,n+p}} \coloneqq \{ X\in \mathcal{M}^{\mathrm{tight}}_{g,n,p}(\mathbf{L},\mathbf{K})^{[\Gamma']_{g,n+p}}: \gamma_1',\ldots,\gamma_k'\text{ tight}\}.
\end{align*}

Finally, if $\Gamma' \in \mathcal{C}^{\Gamma}_{g,n+p}$ is a multicurve on $\Sigma_{g,n+p}$ that belongs to $[\Gamma]_{g,n}$, then the same holds for $h \cdot \Gamma' \in \mathcal{C}^{\Gamma}_{g,n+p}$ when $h \in \lambda_{g,n,p}^{-1}(\mathrm{Stab}_{g,n}(\Gamma))$, so we may construct the quotient $\mathcal{C}^{\Gamma}_{g,n+p}/ \lambda_{g,n,p}^{-1}(\mathrm{Stab}_{g,n}(\Gamma))$.
Our first lemma expresses the intuitive fact that a typical surface with a distinguished extended homotopy class of multicurves also contains a distinguished tight multicurve.
\begin{lemma}\label{lem:distinguishedtight}
    The mapping \eqref{eq:forgethomotopy} gives rise to an injection
    \begin{align*}
        \bigsqcup_{\Gamma'\in \mathcal{C}^{\Gamma}_{g,n+p}/ \lambda_{g,n,p}^{-1}(\mathrm{Stab}_{g,n}(\Gamma))} \hat{\mathcal{M}}^{\mathrm{tight}}_{g,n,p}(\mathbf{L},\mathbf{K})^{[\Gamma']_{g,n+p}} \longrightarrow \mathcal{M}^{\mathrm{tight}}_{g,n,p}(\mathbf{L},\mathbf{K})^{[\Gamma]_{g,n}}
    \end{align*}
    with image of full Weil-Petersson measure.
\end{lemma}
\begin{proof}
    The mapping is well-defined, because for an element $h\in \lambda_{g,n,p}^{-1}(\mathrm{Stab}_{g,n}(\Gamma))$ the moduli spaces $\hat{\mathcal{M}}^{\mathrm{tight}}_{g,n,p}(\mathbf{L},\mathbf{K})^{[\Gamma']_{g,n+p}}$ and $\hat{\mathcal{M}}^{\mathrm{tight}}_{g,n,p}(\mathbf{L},\mathbf{K})^{[h\cdot \Gamma']_{g,n+p}}$ have identical image under \eqref{eq:forgethomotopy}.
    By Proposition~\ref{tighten_multicurve}, for each $X \in \mathcal{M}^{\mathrm{tight}}_{g,n,p}(\mathbf{L},\mathbf{K})^{[\Gamma]_{g,n}}$ there exists a length-minimizing multi-curve $\Gamma' \in \mathcal{C}_{g,n+p}^\Gamma$ whose curves are simple and pairwise disjoint.
    Moreover $\Gamma'$ is tight unless $X$ supports a pair of disjoint simple closed geodesics of equal length.
    The subset of hyperbolic surfaces having such a pair in $\mathcal{M}^{\mathrm{tight}}_{g,n,p}(\mathbf{L},\mathbf{K})^{[\Gamma]_{g,n}}$ corresponds to a countable union of co-dimension 1 sub-manifolds. 
    In particular, its complement, comprising of the surfaces with tight multi-curve $\Gamma'$, in $\mathcal{M}^{\mathrm{tight}}_{g,n,p}(\mathbf{L},\mathbf{K})^{[\Gamma]_{g,n}}$ is of full Weil-Petersson measure.
    If $X$ belongs to the latter, then $\Gamma'\in\mathcal{C}^{\Gamma}_{g,n+p}$ is tight and $(X,\Gamma')\in \hat{\mathcal{M}}^{\mathrm{tight}}_{g,n,p}(\mathbf{L},\mathbf{K})^{[\Gamma']_{g,n+p}}$ is a hyperbolic surface with tight distinguished homotopy class of curves $\Gamma'$.
    This proves the claim about injectivity and the image.
\end{proof}

We may now consider the length-restricted version 
\begin{align*}
\hat{\mathcal{M}}^{\mathrm{tight}}_{g,n,p}(\mathbf{L},\mathbf{K})^{[\Gamma']_{g,n+p}}(\mathbf{x})\subset \mathcal{M}_{g,n+p}(\mathbf{L},\mathbf{K})^{\Gamma'}(\mathbf{x})
\end{align*}
for $\mathbf{x}\in\mathbb{R}_+^k$.
Suppose that $\Sigma_{g,n+p}\setminus \Gamma' = \bigsqcup_{i=1}^q \Sigma_{g_i,n_i+k_i+p_i}$, where $n_i$ and $k_i$ are as before and $p_i$ is the number of additional boundaries ending up in the $i$th component.
Accordingly, if $\mathbf{K}\in \mathbb{R}_+^{p}$, we denote by $\mathbf{K}^{(i)} \in \mathbb{R}_+^{p_i}$ the associated length vector of these additional boundaries.
The central new ingredient in the proof of Proposition~\ref{prop:integrationformula} is the following observation.

\begin{lemma}\label{lem:tightcomponents}
Let $X\in \mathcal{M}^{\mathrm{tight}}_{g,n,p}(\mathbf{L},\mathbf{K})^{[\Gamma']_{g,n+p}}(\mathbf{x})$ and denote the $q$-tuple of hyperbolic surface obtained after cutting $X$ along $\Gamma'$ by $X_1,\ldots,X_q$.
Then $\Gamma'$ is tight in $X$ if and only if for each $i=1,\ldots,k$, $X_i$ is a hyperbolic surface with $n_i+k_i$ tight boundaries.
\end{lemma}
\begin{proof}
    The ``only if'' part is straightforward.
    Indeed, if $\partial$ is one of the first $n_i$ boundaries of $X_i$ then the extended homotopy class $[\partial]_{\Sigma_{g_i,n_i+k_i}}$ is naturally contained in the extended homotopy class $[\partial]_{\Sigma_{g,n}}$ of the original surface, so the tightness of $\partial$ in $X_i$ follows from its tightness in $X$.
    Similarly, if $\partial$ is one of the $k_i$ boundaries corresponding to a side of the curve $\gamma_j'$, then $[\partial]_{\Sigma_{g_i,n_i+k_i}}$ is naturally contained in $[\gamma_j']_{\Sigma_{g,n}}$, so tightness of $\partial$ in $X_i$ follows from tightness of the curve $\gamma_j'$ in $X$.

    For the other direction, if we assume that $X_1,\ldots,X_q$ have tight boundaries, we need to demonstrate that a length-minimizing curve in one of the relevant homotopy classes of $\Sigma_{g,n}$ cannot intersect $\Gamma'$.
    For this we use the following topological fact (see e.g.\ \cite[\S 3]{hass1985}): if $\gamma$ is a simple closed curve in $\Sigma_{g,n}$ that intersects a multi-curve $\Gamma'$ transversally and is freely homotopic to a curve that is disjoint from $\Gamma'$, then $\gamma$ and $\Gamma'$ form at least one bigon, i.e.\ a disk embedded in $\Sigma_{g,n}$ whose boundary consists of a segment $a$ of $\gamma$ and a segment $b$ of $\Gamma'$ that meet at both endpoints and whose interior is disjoint from $\gamma$ and $\Gamma'$.
    Let now $\gamma$ be a length-minimizing geodesic in $[\partial]_{\Sigma_{g,n}}$ for one of the $n$ boundaries $\partial$ or $[\gamma_i]_{\Sigma_{g,n}}$ for one of the curves of $\Gamma'$.
    By the argument in the proof of Lemma~\ref{lem:distinguishedtight} we know that $\gamma$ is simple and, by definition, is freely homotopic in $\Sigma_{g,n}$ to a curve disjoint from $\Gamma'$.
    So if $\gamma$ intersects $\Gamma'$, a bigon must exist.
    This bigon is contained in one of the surfaces $X_j$ (after capping the last $p_j$ boundaries) with the segment $b$ on one of the tight boundaries $\partial$ and $a$ starting and ending on the boundary. 
    By the tightness of $\partial$, $a$ is strictly longer than $b$.
    Hence, there would be a shorter curve in $X$ obtained from $\gamma$ by replacing the segment $a$ by $b$, living in the same extended free homotopy class as $\gamma$.
    Since this is in contradiction with $\gamma$ being length-minimizing, $\gamma$ must be disjoint from $\Gamma'$.
\end{proof}

A direct consequence at the level of moduli spaces is the following.

\begin{lemma}\label{lem:tightmodulifactorization}
    The subset $\hat{\mathcal{M}}^{\mathrm{tight}}_{g,n,p}(\mathbf{L},\mathbf{K})^{[\Gamma']_{g,n+p}}(\mathbf{x})\subset \mathcal{M}_{g,n+p}(\mathbf{L},\mathbf{K})^{\Gamma'}(\mathbf{x})$ is invariant under the torus action (twisting along the curves $\gamma_1',\ldots,\gamma_k'$) and the image under the symplectomorphism \eqref{eq:cuttingsymplectomorphism} of its quotient space $\hat{\mathcal{M}}^{\mathrm{tight}}_{g,n,p}(\mathbf{L},\mathbf{K})^{[\Gamma']_{g,n+p}*}(\mathbf{x})$ is 
    \begin{align*}
        \prod_{i=1}^q \mathcal{M}^{\mathrm{tight}}_{g,n_i+k_i,p_i}(\mathbf{L}^{(i)},\mathbf{x}^{(i)},\mathbf{K}^{(i)}).
    \end{align*}
\end{lemma}
\begin{proof}
    Since the $q$-tuple of hyperbolic surfaces obtained after cutting is invariant under twisting along $\Gamma'$, Lemma~\ref{lem:tightcomponents} implies that tightness of $\Gamma'$ is invariant under twisting as well. 
    The identification of the image under the symplectomorphism follows directly from Lemma~\ref{lem:tightcomponents} as well.
\end{proof}

This brings us to the proof of our proposition.
\begin{proof}[Proof of Proposition~\ref{prop:integrationformula}]
We follow the reasoning of \cite[Sec.~8]{Mirzakhani_polynomial_volumes}, which relates
\begin{align*}
    \int_{\mathcal{M}_{g,n,p}^{\mathrm{tight}}(\mathbf{L},\mathbf{K})} F^\Gamma_{\mathrm{tight}}(X) \mathrm{d}\mu_{g,n+p}^{\mathrm{WP}}(X) &= \int_{\mathcal{M}_{g,n,p}^{\mathrm{tight}}(\mathbf{L},\mathbf{K})^{[\Gamma]_{g,n}}} f(\vec{\ell}_X^{\;\mathrm{tight}}(\Gamma)) \mathrm{d}\mu_{g,n+p}^{\mathrm{WP}}(X).
\end{align*}
Lemma~\ref{lem:distinguishedtight} implies that this equals
\begin{align*}
    \sum_{\Gamma'\in \mathcal{C}^{\Gamma}_{g,n+p}/ \lambda_{g,n,p}^{-1}(\mathrm{Stab}_{g,n}(\Gamma))} \int_{\hat{\mathcal{M}}^{\mathrm{tight}}_{g,n,p}(\mathbf{L},\mathbf{K})^{[\Gamma']_{g,n+p}}} f(\vec{\ell}_X(\Gamma')) \mathrm{d}\mu_{g,n+p}^{\mathrm{WP}}(X),
\end{align*}
while Lemma~\ref{lem:tightmodulifactorization} in turn allows us to factorize the integral into
\begin{align*}
    \sum_{\Gamma'\in \mathcal{C}^{\Gamma}_{g,n+p}/ \lambda_{g,n,p}^{-1}(\mathrm{Stab}_{g,n}(\Gamma))} 
    \int_{\mathbb{R}_+^k} x_1\cdots x_k f(\mathbf{x}) \prod_{i=1}^q \int_{\mathcal{M}^{\mathrm{tight}}_{g,n_i+k_i,p_i}(\mathbf{L}^{(i)},\mathbf{x}^{(i)},\mathbf{K}^{(i)})}
    \mathrm{d}\mu_{g_i,n_i+k_i+p_i}^{\mathrm{WP}}(X_i).
\end{align*}
The inner integrals are nothing but the tight Weil-Petersson volumes $T_{g_i,n_i+k_i,p_i}(\mathbf{L}^{(i)},\mathbf{x}^{(i)},\mathbf{K}^{(i)})$.
Observing that there is precisely one element of $\mathcal{C}^{\Gamma}_{g,n+p}/ \lambda_{g,n,p}^{-1}(\mathrm{Stab}_{g,n}(\Gamma))$ for each partition of the $I_1 \sqcup \cdots \sqcup I_q = \{1,\ldots,p\}$ of the last $p$ boundaries over the $q$ components of $\Sigma_{n,p} \setminus \Gamma$, the claimed formula follows.
\end{proof}

The integration formula takes on a simpler form once expressed in terms of generating functions and restricting to the case where the last $p$ boundaries are cusps,
    \begin{align}\label{integration_formula_version_generating_function}
        \sum_{p=0}^\infty \frac{\mu^p}{p!} \int_{\mathcal{M}_{g,n,p}^{\mathrm{tight}}(\mathbf{L},\mathbf{0}^{p})} F^\Gamma_{\mathrm{tight}}(X) \mathrm{d}\mu_{g,n+p}^{\mathrm{WP}}(X) = C_{\Gamma}\int_{\mathbb{R}_+^k} x_1\cdots x_k f(\mathbf{x}) \prod_{i=1}^q T_{g_i,n_i}(\mathbf{L}^{(i)},\mu) \mathrm{d}x_1\cdots \mathrm{d}x_k.
    \end{align}

\section{Estimates of tight volumes and intersection numbers} \label{estimates_tight_volumes}

This section is dedicated to proving all the estimates we need to prove Theorem~\ref{main_theorem}. In the first subsection~\ref{estimates_intersection_numbers} we will provide new inequalities on intersection numbers, while in subsection~\ref{estimates_of_polynomial} we give estimates of the polynomial 
\begin{align*}
\mathcal{P}_{g,n}\bigg(\displaystyle \mathbf{L},\frac{M_1(\mu)}{M_0(\mu)},\cdots,\frac{M_{3g-3+n}(\mu)}{M_0(\mu)}\bigg).
\end{align*}
The following notation will be practical in this section. For $\boldsymbol{x} = (x_1,\cdots,x_r) \in \mathbb{N}^{r}$, we denote $|\boldsymbol{x}| = \somme{i=1}{r}{x_i}$. For $I \subset \{1,\cdots,r\}$, we define $\boldsymbol{x}_I = (x_i)_{i \in I}$. We also use the notation $\boldsymbol{x}_{\ge i} = (x_j)_{i \le j \le r}$ and $\boldsymbol{x}_{\le i} = (x_j)_{0 \le j \le i}$.

\subsection{Estimates of intersection numbers}\label{estimates_intersection_numbers}
First we mention the result \cite[Corollary 5.11]{Intersection_numbers}  which gives for any $p_1,\dots,p_k \ge 2$:
\begin{align}\label{intersection_numbers}
    \langle \tau_{p_1}\dots\tau_{p_k}\tau_{2}^{3g-3+k - |\mathbf{p}|}\rangle_{g} \underset{g \to +\infty}{\sim }\frac{15^k g^{2k-|\mathbf{p}|}}{\prod_{i=1}^{k}(2p_i +1)!!}\bigg(\frac{25}{24}\bigg)^{g}\frac{2^{g-1}\sqrt{3/5}(3g-3)!((g-1)!)^2}{\pi^2 (5g-5)(5g-3)}.
\end{align}
The next proposition gives a uniform comparison between intersection numbers.
\begin{proposition}\label{uniform_bound_on_intersection_numbers}
For any $g \ge 2$ and $p_1,\cdots,p_k,q_1,\cdots,q_r \ge 1$ such that $|\mathbf{p}|+|\mathbf{q}|  \le 3g-3$, we have 
\begin{align}\label{simple_bound_intersection_numbers}
    \frac{\langle \tau_{q_1+1}\dots \tau_{q_r+1}\tau_{p_1+1}\dots\tau_{p_k+1}\tau_{2}^{3g-3 - |\mathbf{p}|-|\mathbf{q}|}\rangle_{g}}{(3g-3-|\mathbf{p}|-|\mathbf{q}|)!} &\le  \frac{\langle\tau_{p_1}\dots\tau_{p_k}\tau_{2}^{3g-3 - |\mathbf{p}|}\rangle_{g}}{(3g-3-|\mathbf{p}|)!}\cdot \frac{ 3^{|\mathbf{q}|}(15g)^r}{\displaystyle \prod_{i=1}^{r}(2q_i+3)!!}.
\end{align}
\end{proposition}
\begin{proof}
    Fix $g \ge 2$ and $p_1,\cdots,p_k \ge 1$ such that $|\mathbf{p}| \le 3g-3$. We only prove the result for $r=1$. An induction on $r$ gives the general result. Let us prove the result by induction on $q_1$. The case $q_1 = 2$ is obvious. Let us suppose that the result holds for $q_1 \ge 1$. We write the formulation of the Virasoro constraints introduced originally in equation $(7.13)$ of \cite{VERLINDE1991457}: 
\begin{align*}
    \bigg \langle \tau_{k+1}\prod_{i=1}^{n}\tau_{k_i} \bigg \rangle = &\frac{1}{(2k+3)!!}\bigg(\somme{j=1}{n}{\frac{(2k+2k_j+1)!!}{(2k_j-1)!!}\bigg \langle \tau_{k+k_j}\prod_{\substack{i=1\\i\neq j}}^{n}\tau_{k_i} \bigg \rangle}\\
    &+\frac{1}{2}\somme{\substack{r+s=k-1\\r,s \ge 0}}{}{(2r+1)!!(2s+1)!!\bigg \langle \tau_{r}\tau_{s}\prod_{i=1}^{n}\tau_{k_i} \bigg \rangle}\\&+\displaystyle \frac{1}{2}\somme{\substack{r+s=k-1\\r,s \ge 0}}{}{(2r+1)!!(2s+1)!!\prod_{\substack{I\cup J = \{1,\dots,n\}\\I\cap J= \emptyset}}\bigg \langle \tau_{r}\prod_{i\in I}^{}\tau_{k_i} \bigg \rangle\bigg \langle \tau_{s}\prod_{i\in J}^{}\tau_{k_i} \bigg \rangle}\bigg).
\end{align*}
Applying this equation to $\langle \tau_{q_1+1}\tau_{p_1+1}\dots\tau_{p_k+1}\tau_{2}^{3g-3 - |\mathbf{p}|-q_1}\rangle_{g}$, we can give a lower bound for this term by only keeping the terms for which $k_j = 2$ in the first sum of the right-hand side of the equation. We deduce
\begin{align}\label{to_use}
    &\langle \tau_{q_1+1}\tau_{p_1+1}\dots\tau_{p_k+1}\tau_{2}^{3g-3 - |\mathbf{p}|-q_1}\rangle_{g} \\
    &\qquad\qquad\ge \frac{(2q_1+5)!!}{3(2q_1+3)!!}(3g-3-|\mathbf{p}|-q_1)\langle \tau_{q_1+2}\tau_{p_1+1}\dots\tau_{p_k+1}\tau_{2}^{3g-3 - |\mathbf{p}|-(q_1+1)}\rangle_{g}\nonumber
\end{align}
This last equation and the induction gives
\begin{align*}
    \langle \tau_{q_1+2}\tau_{p_1+1}\dots\tau_{p_k+1}\tau_{2}^{3g-3 - |\mathbf{p}|-(q_1+1)}\rangle_{g} &\le \frac{(2q_1+3)!!}{(2q_1+5)!!}\frac{3}{(3g-3-|\mathbf{p}|-q_1)}\cdot\frac{3^{q_1}(15g)}{(2q_1+3)!!}\\& \hspace{0.5cm}\frac{(3g-3-|\mathbf{p}|-q_1)!}{(3g-3-|\mathbf{p}|)!}\langle \tau_{p_1+1}\dots\tau_{p_k+1}\tau_{2}^{3g-3 - |\mathbf{p}|}\rangle_{g} \\
    &\mkern-150mu\le\frac{3^{q_1+1}(15g)}{(2q_1+5)!!} \frac{(3g-3-|\mathbf{p}|-(q_1+1))!}{(3g-3-|\mathbf{p}|)!}\langle \tau_{p_1+1}\dots\tau_{p_k+1}\tau_{2}^{3g-3 - |\mathbf{p}|}\rangle_{g}.
\end{align*}
Dividing the last inequality by $(3g-3-|\mathbf{p}|-(q_1+1))!$, this concludes the proof.
\end{proof}
\subsection{Estimates of $T_{g,n}(\mathbf{L},\mu)$}\label{estimates_of_polynomial}
Let us start by stating the main result of this subsection. This allows us to do computations using Proposition~\ref{prop:integrationformula} and using estimates for $T_{g,n}(\mu)$ (see Proposition~\ref{Estimate_derivative_Pg0}).
\begin{proposition}\label{estimate_boundary}
For any $\mu_g \underset{g \to +\infty}{\rightarrow}\mu_c$ such that $\mu_c - \mu_g = o(g^{-2})$ and any $n \ge 0$, we have
    \begin{align*}
        T_{g,n}\bigg(\bigg(-\frac{M_1}{3M_0}\bigg)^{\frac{1}{2}}\mathbf{L},\mu_g\bigg) = \bigg(\prod_{i=1}^{n}\frac{\sinh(L_i)}{L_i}\bigg)T_{g,n}(\mu_g)\,\bigg(1+o_{n,\mathbf{L}}(1)\bigg),
    \end{align*}
   where $o_{n,\mathbf{L}}(1)$ is uniform in $\mathbf{L}\in K$ if $K \subset \mathbb{R}_{+}^{n}$ is compact.\\
   Moreover, for $K \subset \mathbb{R}^n$ compact, there exists $g_0 \ge 0$ such that for $g$ large enough
   \begin{align*}
      \forall \mathbf{L} \in K, \text{ } \forall g^{'}
\in \{g_0,\cdots,g\}, \text{ } T_{g^{'},n}\bigg(\bigg(-\frac{M_1}{3M_0}\bigg)^{\frac{1}{2}}\mathbf{L},\mu_g\bigg) \le 2e^{\somme{i=1}{n}{L_i}}T_{g^{'},n}(\mu_g), 
\end{align*}
where the constant $g_0$ depends on the choice of $(\mu_g)$, $n$ and $K$.
\end{proposition}

\noindent First we establish asymptotics for the quantities $M_0(\mu),M_1(\mu), R(\mu)$ as $\mu \to \mu_c$ and determine the exact value of $\mu_c$. 
\begin{proposition}\label{asymptotics_M1_M0_R}
    We have the following asymptotics
    \begin{align*}
       & M_{0}(\mu) \underset{\mu \to \mu_c}{\sim} \bigg(\frac{8\pi^2 J_1(j_0)}{j_0}(\mu_c -\mu)\bigg)^{\frac{1}{2}},\\
       & M_{1}(\mu_c) =-\frac{4\pi^2 J_1(j_0)}{j_0},\\
       & R(\mu_c) =\frac{j_0^2}{8\pi^2},\\
       & \mu_c = \frac{j_0J_1(j_0)}{4\pi^2},
    \end{align*}
    where $j_0$ is the first zero of the Bessel function of the first kind $J_0$.
\end{proposition}
\begin{proof}
From \eqref{moments} and using the fact that $M_0(\mu_{c}) = 0$, we deduce that $J_0(2\pi\sqrt{2R(\mu_c)}) = 0$. It follows that $R(\mu_{c}) = \frac{j_0^2}{8\pi^2}$. Still using \eqref{moments}, we get $ M_1(\mu_{c}) = \displaystyle \frac{\partial^2 Z}{\partial^2 r}\bigg(R(\mu_{c}),\mu_{c}\bigg) = -\frac{\sqrt{2}\pi}{\sqrt{R(\mu_c)}}J_1(2\pi\sqrt{2R(\mu_c)}) = -\frac{4\pi^2 J_1(j_0)}{j_0}$.\\
 We also write
\begin{align*}
    M_0(\mu) = \frac{\partial Z}{\partial r}\bigg(R(\mu),\mu\bigg) = \frac{\partial Z}{\partial r}\bigg(R(\mu),\mu_{c}\bigg),
\end{align*}
where the second equality follows from the definition in \eqref{def_Z}. Thus we can rewrite this as
\begin{align} \label{intermediate_equation0}
    M_0(\mu) = \frac{\partial Z}{\partial r}\bigg(R(\mu_{c}),\mu_{c}\bigg) + \frac{\partial^2 Z}{\partial^2 r}\bigg(R(\mu_{c}),\mu_{c}\bigg)\bigg(R(\mu)-R(\mu_c)\bigg) + o\bigg(R(\mu)-R(\mu_c)\bigg).
\end{align}
 Using \eqref{def_Z}, we can also write 
\begin{align}\label{other_intermediate_equation1}
    \mu - \mu_c &=Z(R(\mu_c),\mu_c)-Z(R(\mu),\mu_c)\\&\nonumber = \frac{1}{2}\frac{\partial^2 Z}{\partial^2 r}\bigg(R(\mu_{c}),\mu_{c}\bigg)\bigg(R(\mu_c) - R(\mu)\bigg)^2 +o\bigg(R(\mu_c) - R(\mu)\bigg)^2.
\end{align}
Putting together \eqref{intermediate_equation0} and \eqref{other_intermediate_equation1}, we deduce
\begin{align*}
    M_0(\mu) \underset{\mu \to \mu_c}{\sim}\bigg(\frac{8\pi^2 J_1(j_0)}{j_0}(\mu_c -\mu)\bigg)^{\frac{1}{2}}
\end{align*}
Finally, writing $\displaystyle Z(R(\mu_c),\mu_c) = 0$, we deduce that $\mu_c = \displaystyle \frac{\sqrt{R(\mu_c)}}{\sqrt{2}\pi}J_1(2\pi\sqrt{2R(\mu_c)})  =\displaystyle \frac{j_0J_1(j_0)}{4\pi^2}$.
\end{proof}
\begin{remark1}
    This value for $\mu_c$ is consistent with \cite[Theorem $6.1$]{manin1999invertible}.
\end{remark1}
\noindent Using \eqref{moments} and the fact that $|J_k(x)| \le 1$ for every $k \ge 0$ and $x \ge 0$ and Proposition~\ref{asymptotics_M1_M0_R}, there exists an absolute constant $a>1$ such that for all $k \geq 0$ and any $\mu \in [0,\mu_c)$, we have
\begin{align}\label{bound_Mk}
    |M_k(\mu)| \le a^k \text{ and } |M_1(\mu)| \in [a^{-1},a].
\end{align}
For $g\ge 2$, let us recall the expression of $\mathcal{P}_{g,0}(\mathbf{m})$ given in \eqref{Pg0_expression},
\begin{align*}
    \mathcal{P}_{g,0}(\mathbf{m}) = \somme{\substack{d_2,d_3\cdots \ge 0\\ \somme{k=2}{3g-2}{(k-1)d_k} = 3g-3}}{}{\langle \tau_{2}^{d_2}\cdots\tau_{3g-2}^{d_{3g-2}}\rangle_g \prod_{k=2}^{3g-2}\frac{(-m_{k-1})^{d_k}}{d_k!}}.
\end{align*}
We want to give asymptotics for $\displaystyle \mathcal{P}_{g,n}\bigg(\mathbf{L},\frac{M_1(\mu)}{M_0(\mu)},\cdots,\frac{M_{3g-3+n}(\mu)}{M_0(\mu)}\bigg)$. In the following, we use the notation $\mathbf{m} = (m_1,\dots,m_{3g-3+n})$ and $\mathbf{M}(\mu) = \displaystyle \bigg(\frac{M_1(\mu)}{M_0(\mu)},\cdots,\frac{M_{3g-3+n}(\mu)}{M_0(\mu)}\bigg)$. We do not specify the dependence in $\mu$ when it is clear from the context. For $\mathbf{p} = (p_1,\cdots,p_k) \in (\mathbb{N}^{*})^{k}$, we define $\displaystyle \frac{\partial}{\partial m_{\mathbf{p}}} = \frac{\partial^{k}}{\partial m_{p_1}\cdots \partial m_{p_k}}$. For $x\in \mathbb{N}$ and $\mathbf{p} \in \mathbb{N}^k$, we write $(x,\mathbf{p}):= (x,p_1,\cdots,p_k)$.\\
As we will see in this Section, it will be convenient to rescale the lengths $\mathbf{L}$ by $\bigg(\frac{-M_1}{3M_0}\bigg)^{-\frac{1}{2}}$. Thus for $p_1,\cdots,p_k\ge 1$ we introduce
\begin{align}\label{def_rescaled_columes}
    \mathcal{T}_{n,g,\mathbf{p}}(\mathbf{L},\mathbf{m}):=\frac{\partial \mathcal{P}_{g,n}}{\partial m_{\mathbf{p}}}\bigg(\sqrt{-m_1/3}\,\mathbf{L},\mathbf{m}\bigg), 
\end{align}
and for $q_1,\cdots,q_n \ge 0$, we introduce the coefficient associated to $L_1^{2q_1}\cdots L_n^{2q_n}$
\begin{align}\label{def_coeff_c}   \alpha_{n,g,\mathbf{p},\mathbf{q}} :=   [L_1^{2q_1}\cdots L_n^{2q_n}]\mathcal{T}_{n,g,\mathbf{p}}(\mathbf{L},\mathbf{m})= [L_1^{2q_1}\cdots L_n^{2q_n}]\frac{\partial \mathcal{P}_{g,n}}{\partial m_{\mathbf{p}}}\bigg(\sqrt{-m_1/3}\,\mathbf{L},\mathbf{m}\bigg).
\end{align}\\
When $n=0$ we do not specify the dependence in $n$. Same if $\mathbf{q} = (0,\cdots,0)$ or if $\mathbf{p}=\emptyset$.\\
Let us also give crude estimates on $\mathcal{T}_{n,g}(\mathbf{L},\mathbf{M})$ and $\mathcal{T}_{n,g}(\mathbf{M}) = \alpha_{n,g}(\mathbf{M})$.\\
\begin{proposition}\label{crude_bound} For any $g,n \ge 0$ such that $2g+n-3 \ge 0$ and any compact $K \subset (\mathbb{R}_{+})^{n}$ we have 
    \begin{align*}
        \forall \mu \in [0,\mu_c),\text{ }\forall \mathbf{L}\in K, \hspace{0.3cm} |\mathcal{T}_{n,g}(\mathbf{L},\mathbf{M})| \le  C_{n,g,K}\bigg(\frac{-M_1}{M_0}\bigg)^{3g-3+n},
    \end{align*}
    where $C_{n,g,K} > $ is a constant that depends on $n,g$ and $K$.
    Moreover, we have
    \begin{align*}
        \alpha_{n,g}(\mathbf{M}) \underset{\mu \to \mu_c}{\sim} C_{n,g}\bigg(\frac{-M_1}{M_0}\bigg)^{3g-3+n},
    \end{align*}
    where $C_{n,g} > 0$ only depends on $n$ and $g$.
\end{proposition}
\begin{proof}
    By Theorem~\ref{polynomial_recursion}, the polynomial $\mathcal{P}_{n,g}\bigg(\sqrt{\sigma}\mathbf{L},(\sigma m_1,\cdots,\sigma^{3g-3+n}m_{3g-3+n})\bigg)$ is homogeneous of degree $3g-3+n$ in $\sigma$. Replacing $\mathbf{m}$ by $\mathbf{M}$ and noticing that due to \eqref{bound_Mk} there exists a constant $c>1$ such that $|M
    _k/M_0| \leq |c\,M_1/M_0|^k$ for all $\mu\in[0,\mu_c)$ and $k\geq 1$, the claimed bound follows. The second part of the proposition follows from the fact that $[m_1^{3g-3+n}]\alpha_{n,g} \neq 0$. This can be checked by induction using Theorem~\ref{polynomial_recursion}. Then replacing $\mathbf{m}$ by $\mathbf{M}$ and letting $\mu \to \mu_c$, the dominant term is the one corresponding to $m_1^{3g-3+n}$. This concludes the proof with $C_{n,g} = (-1)^{3g-3+n}[m_1^{3g-3+n}]\alpha_{n,g}$.
\end{proof}
\noindent 
The purpose of this section is to provide precise asymptotics and estimates on the coefficients $\alpha_{n,g,\mathbf{p},\mathbf{q}}(\mathbf{M})$.
Our main tool will be Theorem~\ref{polynomial_recursion} since this allows us to go from $\mathcal{P}_{g,n}$ to $\mathcal{P}_{g,n+1}$.  
For $g \ge 2$ and $(p_1,\cdots,p_k)\in (\mathbb{N}^{*})^k$ such that $|\mathbf{p}|\le 3g-3$, we introduce $\varphi(g,\mathbf{p})$ for the combination
\begin{align}\label{def_phi_gp}
    \varphi(g,\mathbf{p}) = (-1)^{k}\frac{\langle \tau_{p_1+1}\cdots \tau_{p_k +1 }\tau_2^{3g-3-|\mathbf{p}|}\rangle_{g}\displaystyle\bigg(\frac{-M_1}{M_0}\bigg)^{3g-3-|\mathbf{p}|}}{(3g-3-|\mathbf{p}|)!}
\end{align}
that we will encounter later.
Furthermore, for $n \ge 0$, we define
\begin{align}\label{def_phi}
    \varphi(n,g,\mathbf{p}) &= \varphi(g,\mathbf{p})\left(\frac{-M_1}{M_0}\right)^{n}(5g)^n \\&\nonumber= (-1)^{k}\frac{\langle \tau_{p_1+1}\cdots \tau_{p_k +1 }\tau_2^{3g-3-|\mathbf{p}|}\rangle_{g}\displaystyle\bigg(\frac{-M_1}{M_0}\bigg)^{3g-3+n-|\mathbf{p}|}}{(3g-3-|\mathbf{p}|)!}(5g)^n.
\end{align}
When $n = 0$ or $\mathbf{p}=\emptyset$, we again suppress the argument to ease the notation. A direct application of Proposition~\ref{uniform_bound_on_intersection_numbers} gives
\begin{align}\label{bound_phi_with_phi_p0}
    |\varphi(n,g,\mathbf{p})| \le \bigg(\frac{-M_1}{M_0}\bigg)^{n-|\mathbf{p}|}\frac{(5g)^{k+n}3^{k+|\mathbf{p}|}}{\prod_{i=1}^{k}(2p_i+3)!!}\varphi(g).
\end{align}
For $n$, $\mathbf{p}$ and $q \ge 1$ fixed, using \eqref{intersection_numbers} and \eqref{def_phi}, as $g \to +\infty$, we let the reader verify that we have
\begin{align}
    &\varphi(n,g,(1,\mathbf{p})) \underset{g\to +\infty}{\sim} -\bigg(\frac{-M_1}{M_0}\bigg)^{-1}(3g)\varphi(n,g,\mathbf{p}),\label{few_equations_phi}\\
    &     \varphi\bigg(n,g,(q,\mathbf{p})\bigg) \underset{g \to +\infty}{\sim} -\frac{2^{q+1}(q+1)!\displaystyle \bigg(-\frac{M_1}{3M_0}\bigg)^{-q-1} }{(2q+3)!}\varphi(n+1,g,\mathbf{p}),\label{few_equations_phi2}\\
    &\varphi(n,g,\mathbf{p}) \sim C_{n,\mathbf{p}}\bigg(\frac{-M_1}{M_0}\bigg)^{n-|\mathbf{p}|}g^{2k+n}\varphi(g),\label{asymptotic_phi_n_p_fixed}
\end{align}
where $C_{n,\mathbf{p}} > 0$ only depends on $n$ and $\mathbf{p}$.\\
First, we focus on $\alpha_{g,\mathbf{p}}(\mathbf{M})=\displaystyle \frac{\partial \mathcal{P}_{g,0}}{\partial m_{\mathbf{p}}}(\mathbf{M})$. Using Theorem~\ref{polynomial_recursion}, we write
    \begin{align}\label{expression_derivative_Pg0}
        \displaystyle \frac{\partial^{}\mathcal{P}_{g,0}}{\partial m_{\mathbf{p}}} = (-1)^{k}\displaystyle \somme{\substack{d_2,d_3\cdots \ge 0\\ \somme{k=2}{3g-2}{(k-1)d_k} = 3g-3-|\mathbf{p}|}}{}{\langle \tau_{p_1+1}\cdots \tau_{p_k+1}\tau_{2}^{d_2}\cdots\tau_{3g-2}^{d_{3g-2}}\rangle_g\displaystyle  \prod_{k=2}^{3g-2}\frac{(-m_{k-1})^{d_k}}{d_k!} }.
    \end{align}
    We show in the next proposition that the sum is concentrated around the term corresponding to $d_2=3g-3-|\mathbf{p}|$. Indeed, replacing $\mathbf{m}$ by $\mathbf{M}$, for $g$ fixed, one can let $\mu \underset{}{\rightarrow} \mu_c$ and note that the sum is equivalent to the term given by $d_2 = 3g-3-|\mathbf{p}|$. We highlight the fact that this is the only reason we need to work with the assumption $\mu_c - \mu_g = o(g^{-2})$.
\begin{proposition}\label{split_sum}
    There exists a function $r(g,\mu)$ such that for every $g\geq 2$, $\mu\in[0,\mu_c)$ and $p_1,\cdots,p_k \ge 1$ such that $|\mathbf{p}| \le 3g-3$, we have
    \begin{align*}
\displaystyle \bigg|\displaystyle\frac{\alpha_{g,\mathbf{p}}(\mathbf{M})}{\varphi(g,\mathbf{p})}-1\bigg| \le r(g,\mu).
    \end{align*}
Moreover, one may choose $r(g,\mu)$ such that $r(g,\mu) \underset{\mu \to \mu_c}{\rightarrow}0$ for any $g\geq 2$ fixed and $r(g,\mu_g) \underset{g \to +\infty}{\rightarrow}0$ when $\mu_c - \mu_g = o(g^{-2})$.
\end{proposition}
\begin{proof}
 Fix $g\ge 2$ and $p_1,\cdots,p_k \ge 1$ such that $|\mathbf{p}| \le 3g-3$. We recall that \eqref{expression_derivative_Pg0} gives
    \begin{align*}
        \displaystyle \alpha_{g,\mathbf{p}} = (-1)^{k}\displaystyle \somme{\substack{d_2,d_3\cdots \ge 0\\ \somme{k=2}{3g-2}{(k-1)d_k} = 3g-3-|\mathbf{p}|}}{}{\langle \tau_{p_1+1}\cdots \tau_{p_k+1}\tau_{2}^{d_2}\cdots\tau_{3g-2}^{d_{3g-2}}\rangle_g\displaystyle  \prod_{k=2}^{3g-2}\frac{(-m_{k-1})^{d_k}}{d_k!} }.
    \end{align*}
Fix $d_2,\cdots,d_{3g-2}$ in the sum. Let $\mathbf{q} = (q_1,\ldots,q_r)$ be so that
\begin{align*}
    (q_1+1,q_2+1,\cdots,q_r+1) = (\underbrace{3,\dots,3}_{d_3 \text{ times}},\cdots,\underbrace{3g-2,\dots,3g-2}_{d_{3g-2} \text{ times}}).
\end{align*}
With this notation, we have $r = \somme{i=3}{3g-2}{d_i}$ and $|\mathbf{q}| = \somme{i=3}{3g-2}{(i-1)d_i}$ and $d_2 = 3g-3-|\mathbf{p}|-|\mathbf{q}|$.
Applying Proposition~\ref{uniform_bound_on_intersection_numbers}, we find the bound
\begin{align}\label{uniform_bound_for_intersection}
     \frac{\langle \tau_{p_1+1}\cdots \tau_{p_k+1}\tau_{2}^{d_2}\cdots\tau_{3g-2}^{d_{3g-2}}\rangle_g}{d_2 !} &\le  \frac{\langle \tau_{p_1+1}\cdots \tau_{p_k+1}\tau_{2}^{3g-3-|\mathbf{p}|}\rangle_g}{(3g-3-|\mathbf{p}|)!} \cdot \frac{3^{\somme{i=3}{3g-2}{(i-1)d_i}} (15g)^{\somme{i=3}{3g-2}{d_i}}}{\prod_{i=3}^{3g-2}(2i+3)!!^{d_i}}.
\end{align}
Replacing $\mathbf{m}$ by $\mathbf{M}$ and using \eqref{uniform_bound_for_intersection}, we have
\begin{align}
    \langle \tau_{p_1+1}\cdots \tau_{p_k+1}\tau_{2}^{d_2}\cdots\tau_{3g-2}^{d_{3g-2}}\rangle_g \prod_{k=2}^{3g-2}\frac{\displaystyle \bigg|\frac{M_{k-1}}{M_0}\bigg|^{d_k}}{d_k!} \le &\frac{\langle \tau_{p_1+1}\cdots \tau_{p_k+1}\tau_{2}^{3g-3-|\mathbf{p}|}\rangle_g}{(3g-3-|\mathbf{p}|)!}\bigg(-\frac{M_1}{M_0}\bigg)^{d_2}\nonumber\\
    &\times  \prod_{k=3}^{3g-2}\frac{3^{(k-1)d_k}(15g)^{d_k} \bigg|{\displaystyle\frac{M_{k-1}}{M_0}}\bigg|^{d_k}}{(2k+3)!!^{d_k}d_k !}.\label{intermediate_equation}
\end{align}

We recall that $d_2 = 3g-3-|\mathbf{p}| - \somme{i=3}{3g-2}{(i-1)d_i}$. Using \eqref{bound_Mk}, we can write
\begin{align*}
   \bigg(-\frac{M_1}{M_0}\bigg)^{d_2}\prod_{k=3}^{3g-2}\bigg|\frac{M_{k-1}}{M_0}\bigg|^{d_k} \le  \bigg(-\frac{M_1}{M_0}\bigg)^{3g-3-|\mathbf{p}|}\prod_{k=3}^{3g-2}(a^2M_0)^{d_k(k-2)}.
\end{align*}
It follows that the expression on the right-hand side of \eqref{intermediate_equation} is bounded by
\begin{align*}
    |\varphi(g,\mathbf{p})| \prod_{k=3}^{3g-2}\frac{3^{(k-1)d_k}(15g)^{d_k}\big(a^2M_0\big)^{d_k(k-2)}}{(2k+3)!!^{d_k}d_k!} \le |\varphi(g,\mathbf{p})| \prod_{k=3}^{3g-2}\frac{\big(a^2gM_0\big)^{d_k(k-2)}}{d_k!},
\end{align*}
where the inequality holds because $3^{k-1}(15g)/(2k+3)!!\le g^{k-2}$.
So we conclude that
\begin{align*}
\bigg|\displaystyle\frac{\alpha_{g,\mathbf{p}}(\mathbf{M})}{\varphi(g,\mathbf{p})}-1\bigg| &\le 
    \somme{\substack{d_2,d_3\cdots \ge 0\\ \somme{k=2}{3g-2}{(k-1)d_k} = 3g-3-|\mathbf{p}|\\d_2 < 3g-3-|\mathbf{p}|}}{}{}\prod_{k=3}^{3g-2}\frac{\big(a^2gM_0\big)^{d_k(k-2)}}{d_k!} \\
    &\le -1+ \prod_{k=3}^{3g-2} \sum_{d_k=0}^\infty\frac{\big(a^2gM_0\big)^{d_k(k-2)}}{d_k!} \\
    &\le r(g,\mu) := \exp\bigg(\somme{k=3}{3g-2}{(a^2gM_0)^{k-2}}\bigg) - 1,
\end{align*}
From Proposition~\ref{asymptotics_M1_M0_R} we see that this choice of $r(g,\mu)$ indeed approaches $0$ for both limits because $g M_0 \to 0$.
\end{proof}

Now, we aim to give an estimate for $\displaystyle \alpha_{n,g,\mathbf{p}}(\mathbf{M})$. To do so, we need a formula for $\displaystyle \alpha_{n+1,g,\mathbf{p}}$ depending on $\displaystyle \alpha_{n,g,\mathbf{p}^{'}}$. By \eqref{recursion_relation}, we can write
    \begin{align}\label{apply_recursion}
       \alpha_{n+1,g}= \somme{p=1}{3g-3+n}{\bigg(m_{p+1}-m_1 m_p}\bigg) \alpha_{n,g,p} -m_1 (2g-2+n) \alpha_{n,g}.
    \end{align}
For any $p_1,\cdots,p_k \ge 1$, we define 
\begin{align}\label{def_A}
    \mathcal{R}_{g,n,1,\mathbf{p}} = &-2m_1 \alpha_{n,g,(1,\mathbf{p})}- \somme{p=2}{3g-3+n}{m_p\alpha_{n,g,(p,\mathbf{p})}}   - (2g-2+n)\alpha_{n,g,\mathbf{p}},
\end{align}
and for $2 \le j \le 3g-3+n+1$ we define
\begin{align}\label{def_B}
   \mathcal{R}_{g,n,j,\mathbf{p}} = \alpha_{n,g,(j-1,\mathbf{p})} -m_1\alpha_{n,g,(j,\mathbf{p})}.
\end{align}
Now, observe that for any $1 \le j \le 3g-3+n+1$ and $p_1,\dots,p_k \ge 0$, we have
\begin{align}
    &\hspace{1.1cm}\frac{\partial}{\partial m_j}\bigg(\somme{p=1}{3g-3+n}{\bigg(m_{p+1}-m_1 m_p}\bigg) \alpha_{n,g,(p,\mathbf{p})} -m_1 (2g-2+n) \alpha_{n,g,\mathbf{p}}\bigg)\\
    &=\nonumber \mathcal{R}_{g,n,j,\mathbf{p}}+\somme{p=1}{3g-3+n}{\bigg(m_{p+1}-m_1 m_p}\bigg) \alpha_{n,g,(p,j,\mathbf{p})} -m_1 (2g-2+n) \alpha_{n,g,(j,\mathbf{p})}.
\end{align}
It follows that we have
\begin{align}\label{expression_derivative_n_1}
    \alpha_{n+1,g,\mathbf{p}} &= \somme{j=1}{k}{\frac{\partial \mathcal{R}_{g,n,p_j,\mathbf{p}_{\le j-1}}}{\partial m_{\mathbf{p}_{\ge j+1}}}}\\
    &\nonumber+\somme{p=1}{3g-3+n}{\bigg(m_{p+1}-m_1 m_p}\bigg) \alpha_{n,g,(p,\mathbf{p})} -m_1 (2g-2+n) \alpha_{n,g,\mathbf{p}}.
\end{align}
For $1 \le j \le k$, we define 
\begin{align}\label{def_A}
\mathcal{A}_{g,n,\mathbf{p},j} := \frac{\partial \mathcal{R}_{g,n,p_j,\mathbf{p}_{\le j-1}}}{\partial m_{\mathbf{p}_{\ge j+1}}} = \left\{\begin{array}{ll}
         \begin{array}{ll}
    &\displaystyle  -2m_1\alpha_{n,g,\mathbf{p}}-2\somme{\substack{r=j+1\\p_r=1}}{k}{\alpha_{n,g,\mathbf{p}_{\widehat{\{r\}}}}}  \\&\displaystyle  - \somme{p=2}{3g-3+n}{m_p\alpha_{n,g,(p,\mathbf{p}_{\widehat{\{j\}}})}} -\somme{\substack{r=j+1\\p_r \ge 2}}{k}{\alpha_{n,g,\mathbf{p}_{\widehat{\{j\}}}}}\\&   \displaystyle - (2g-2+n)\alpha_{n,g,\mathbf{p}_{\widehat{\{j\}}}}
         \end{array}&  \text{ if } p_j = 1.\\
\displaystyle \alpha_{n,g,(p_j -1,\mathbf{p}_{\widehat{\{j\}}})}- m_1 \alpha_{n,g,\mathbf{p}} - \somme{\substack{r=j+1\\p_r=1}}{k}{\alpha_{n,g,\mathbf{p}_{\widehat{\{r\}}}}}& \text{ if } p_j \neq 1.
   \end{array}\right.
\end{align}
We have the following concise formula
\begin{align}\label{recursion_polynomial_with_derivative}
\alpha_{n+1,g,\mathbf{p}} &= \somme{j=1}{k}{\mathcal{A}_{g,n,\mathbf{p},j} }+\somme{p=1}{3g-3+n}{\bigg(m_{p+1}-m_1 m_p}\bigg) \alpha_{n,g,(p,\mathbf{p})} -m_1 (2g-2+n) \alpha_{n,g,\mathbf{p}}.
\end{align}

\begin{proposition}\label{uniform_bound_derivative_and_A}
    For any $\mu_g \underset{g \to +\infty}{\rightarrow}\mu_c$ such that $\mu_c - \mu_g = o(g^{-2})$ and any $n\ge 0$, we have for $g$ large enough
    \begin{align*}
&\forall \mathbf{p}\in (\mathbb{N}^{*})^{k}, \hspace{0.3cm}\big|\alpha_{n,g,\mathbf{p}}(\mathbf{M})\big| \le C_n \bigg( \frac{-M_1}{M_0}\bigg)^{n-|\mathbf{p}|}g^{n+k}A^{k}\varphi(g),\\& \forall\mathbf{p}\in (\mathbb{N}^{*})^{k}, \hspace{0.2cm} \forall j \in \{1,\cdots,k\}, \hspace{0.3cm}\big|\mathcal{A}_{g,n,\mathbf{p},j}(\mathbf{M})\big| \le C^{'}_n \bigg( \frac{-M_1}{M_0}\bigg)^{n+1-|\mathbf{p}|}g^{n+k}A^{k}\varphi(g),
    \end{align*}
where $C_n,C^{'}_n > 0$ only depend on $n$ and $A > 1$ is an absolute constant.
\end{proposition}
\begin{remark1}
    Note that for $\mathbf{p}$ such that $|\mathbf{p}|> 3g-3+n+1$ we have $\alpha_{n,g,\mathbf{p}}(\mathbf{M}) = 0$ and $\mathcal{A}_{g,n,\mathbf{p},j}(\mathbf{M}) = 0$. Thus, in that the case the bound given by the Proposition is useless. However, it will be convenient to have a general formulation that holds for any $\mathbf{p}$.
\end{remark1}
\begin{proof}
    Let us prove the result by induction on $n$ in the case $g \to +\infty$.\\ 
\textbf{Initial case:}\\
For $n = 0$. By proposition~\ref{split_sum}, for $g$ large enough, for any $\mathbf{p}$ such that $|\mathbf{p}|\le 3g-3$, we have
\begin{align*}
    \big|\alpha_{0,g,\mathbf{p}}(\mathbf{M})\big| &\le 2\varphi(g,\mathbf{p})\\
    &\underset{\eqref{bound_phi_with_phi_p0}}{\le} 2\bigg(\frac{-M_1}{M_0}\bigg)^{-|\mathbf{p}|}\frac{(5g)^{k}3^{k+|\mathbf{p}|}}{\prod_{i=1}^{k}(2p_i+1)!!}\varphi(g)\\&\le 2\bigg(\frac{-M_1}{M_0}\bigg)^{-|\mathbf{p}|}g^{k}A^k\varphi(g)\cdot \underbrace{\frac{3^{|\mathbf{p}|}}{\prod_{i=1}^{k}(2p_i+1)!!}}_{\le 1}
    \\&\le C_0\bigg(\frac{-M_1}{M_0}\bigg)^{-|\mathbf{p}|}g^{k}A^k\varphi(g),
\end{align*}
where we have set $C_0 = 2$ and $A = 15$. The bound for $\big|\mathcal{A}_{g,0,\mathbf{p},j}(\mathbf{M})\big|$ will be given in the induction case.\\
\textbf{Induction case:}\\
Suppose that for $g$ large enough, we have
\begin{align*}
    \forall \mathbf{p}\in (\mathbb{N}^{*})^{k}, \hspace{0.3cm}\big|\alpha_{n,g,\mathbf{p}}(\mathbf{M})\big| \le C_n \bigg( \frac{-M_1}{M_0}\bigg)^{n-|\mathbf{p}|}g^{n+k}A^{k}\varphi(g).
\end{align*}
Let us first prove the bound for $\big|\mathcal{A}_{g,n,\mathbf{p},j}(\mathbf{M})\big|$. Since we need to prove the result for $g$ large enough, we suppose in the rest of the proof that $n \le g$. Fix $\mathbf{p}$ such that $|\mathbf{p}| \le 3g-3+n+1$ and $j \in \{1,\cdots,k\}$. Indeed in the case $|\mathbf{p}| > 3g-3+n+1$, we have $\mathcal{A}_{g,n,\mathbf{p},j}(\mathbf{M}) = 0$. First, if $p_j = 1$, by \eqref{def_A} and using the induction property, for $g$ large enough we have the bound 
\begin{align*}
\big|\mathcal{A}_{g,n,\mathbf{p},j}(\mathbf{M})\big| \le C_n A^k\varphi(g)\bigg[&2\bigg(\frac{-M_1}{M_0}\bigg)^{n+1-|\mathbf{p}|}g^{n+k}+2\somme{\substack{r=j+1\\p_r=1}}{k}{\bigg(\frac{-M_1}{M_0}\bigg)^{n-(|\mathbf{p}|-1)}g^{n+k-1}}\\
 &+ \somme{r=2}{3g-3+n}{\bigg|\frac{M_{r}}{M_0}\bigg|\bigg(\frac{-M_1}{M_0}\bigg)^{n-(|\mathbf{p}|-1+r)}g^{n+k}} +\somme{\substack{r=j+1\\p_r\ge 2}}{k}{\bigg(\frac{-M_1}{M_0}\bigg)^{n-(|\mathbf{p}|-1)}g^{n+k-1}}\\&+(2g-2+n)\bigg(\frac{-M_1}{M_0}\bigg)^{n-(|\mathbf{p}|-1)}g^{n+k-1}\bigg].
\end{align*}
Using \eqref{bound_Mk} and the fact that $M_0 \underset{g \to +\infty}{\rightarrow}0$, for $g$ large enough we have 
\begin{align*}
    \displaystyle \somme{r=2}{3g-3+n}{\bigg|\frac{M_{r}}{M_0}\bigg|\bigg(\frac{-M_1}{M_0}\bigg)^{n-(|\mathbf{p}|-1+r)}g^{n+k}} \le \bigg(\frac{-M_1}{M_0}\bigg)^{n+1-|\mathbf{p}|}g^{n+k}.
\end{align*}
We can also write 
\begin{align*}
2\somme{\substack{r=j+1\\p_r=1}}{k}{\bigg(\frac{-M_1}{M_0}\bigg)^{n-(|\mathbf{p}|-1)}g^{n+k-1}} + &\somme{r=2}{3g-3+n}{\somme{\substack{r=j+1\\p_r=p}}{k}{\bigg(\frac{-M_1}{M_0}\bigg)^{n-(|\mathbf{p}|-1)}g^{n+k-1}}} \\
&\le 2k\bigg(\frac{-M_1}{M_0}\bigg)^{n+1-|\mathbf{p}|}g^{n+k-1}\le 8\bigg(\frac{-M_1}{M_0}\bigg)^{n+1-|\mathbf{p}|}g^{n+k},
\end{align*}
where the first inequality follows from the fact that $\displaystyle \somme{\substack{r=j+1\\p_r=1}}{k}1+\somme{r=2}{3g-3+n}{\somme{\substack{r=j+1\\p_r=p}}{k}1} = k$ and the second from the fact that $k \le 3g-3+n \le 4g$. Taking $C^{'}_{n} = 14C_n$ and combining last equations together we obtain the desired bound $\big|\mathcal{A}_{g,n,\mathbf{p},j}(\mathbf{M})\big| \le C^{'}_{n} \bigg(\displaystyle \frac{-M_1}{M_0}\bigg)^{n+1-|\mathbf{p}|}g^{n+k}A^{k}\varphi(g)$. Now if $p_j \ge 2$, by \eqref{def_A} and the induction case we can write
\begin{align*}    \big|\mathcal{A}_{g,n,\mathbf{p},j}(\mathbf{M})\big| \le C_n A^k\varphi(g)\bigg[2\bigg(\frac{-M_1}{M_0}\bigg)^{n+1-|\mathbf{p}|}g^{n+k} + \somme{\substack{r=j+1\\p_r=1}}{k}{\bigg(\frac{-M_1}{M_0}\bigg)^{n+1-|\mathbf{p}|}g^{n+k-1}}\bigg].
\end{align*}
Using the fact that the sum on the right-hand side of the inequality has at most $4g$ terms, we conclude
\begin{align*}
     \forall\mathbf{p}\in (\mathbb{N}^{*})^{k}, \hspace{0.2cm} \forall j \in \{1,\cdots,k\}, \hspace{0.3cm}\big|\mathcal{A}_{g,n,\mathbf{p},j}(\mathbf{M})\big| \le C^{'}_{n} \bigg(\frac{-M_1}{M_0}\bigg)^{n+1-|\mathbf{p}|}g^{n+k}A^k\varphi(g).
\end{align*}
Now let us bound $\big|\alpha_{n+1,g,\mathbf{p}}(\mathbf{M})\big|$. Fix $\mathbf{p}$ such that $ |\mathbf{p}| \le 3g-3+n+1$. Recall from \eqref{recursion_polynomial_with_derivative} that
\begin{align*}
    \alpha_{n+1,g,\mathbf{p}} &= \somme{j=1}{k}{\mathcal{A}_{g,n,\mathbf{p},j} }+\somme{p=1}{3g-3+n}{\bigg(m_{p+1}-m_1 m_p}\bigg)  \alpha_{n,g,(p,\mathbf{p})}  -m_1 (2g-2+n)  \alpha_{n,g,\mathbf{p}} .
\end{align*}
First, we have the bound
\begin{align}\label{bound_term_left}
    \bigg|\somme{j=1}{k}{\mathcal{A}_{g,n,\mathbf{p},j} }(\mathbf{M})\bigg| \le  4C^{'}_{n} \bigg(\frac{-M_1}{M_0}\bigg)^{n+1-|\mathbf{p}|}g^{n+1+k}A^k\varphi(g),
\end{align}
where we used that $k \le 4g$.
Using the induction, \eqref{bound_Mk}, and the fact that $M_0 \underset{g\to +\infty}{\rightarrow}0$, we also write 
\begin{align}\label{bound_term_middle}
    \bigg|\somme{p=1}{3g-3+n}{\bigg(\frac{-M_{p+1}}{M_0}-\frac{M_1M_p}{M_0^2}}\bigg)  \alpha_{n,g,(p,\mathbf{p})} (\mathbf{M})\bigg| &\le a_1 \somme{p=1}{3g-3+n}{C^{}_{n} \bigg(\frac{-M_1}{M_0}\bigg)^{n+2-|\mathbf{p}|-p}g^{n+1+k}A^k\varphi(g)}\\
    &\le a_2 C^{}_{n} \bigg(\frac{-M_1}{M_0}\bigg)^{n+1-|\mathbf{p}|}g^{n+1+k}A^k\varphi(g),\nonumber
\end{align}
where $a_1,a_2 > 0$ are absolute constants.\\
We bound
\begin{align}\label{bound_right_term}
  \bigg|  \frac{-M_1}{M_0}(2g-2+n) \alpha_{n,g,\mathbf{p}} (\mathbf{M})\bigg| \le 3C_n\bigg(\frac{-M_1}{M_0}\bigg)^{n+1-|\mathbf{p}|}g^{n+1+k}A^k\varphi(g).
\end{align}
Combining \eqref{bound_term_left}, \eqref{bound_term_middle} and \eqref{bound_right_term} we have 
\begin{align*}
     \displaystyle \forall \mathbf{p}\in (\mathbb{N}^{*})^{k},\hspace{0.3cm}\big| \alpha_{n+1,g,\mathbf{p}} (\mathbf{M})\big| \le C_{n+1}\bigg(\frac{-M_1}{M_0}\bigg)^{n+1-|\mathbf{p}|}g^{n+1+k}A^k\varphi(g),
\end{align*}
where $C_{n+1} := a_2C_n+3C_n+4C^{'}_n$. This concludes the proof.
\end{proof}

\noindent Now, we give an asymptotic for $ \alpha_{n,g,\mathbf{p}} (\mathbf{M})$ when $n$ and $\mathbf{p}$ are fixed.

\begin{proposition}\label{Estimate_derivative_Pg0}
    For any $\mu_g \underset{g \to +\infty}{\rightarrow}\mu_c$ such that $\mu_c - \mu_g = o(g^{-2})$ and any $p_1,\cdots,p_k \ge 1$, $n \ge 0$, we have 
    \begin{align*}
        \alpha_{n,g,\mathbf{p}} (\mathbf{M}) \underset{g \to +\infty}{\sim} \varphi(n,g,\mathbf{p}).
    \end{align*}
\end{proposition}
\begin{proof}
We do the proof by induction on $n$.\\
\textbf{Initial case:}\\
This is a direct application of Proposition~\ref{split_sum}.\\
\textbf{Induction case:}\\
We define
\begin{align}\label{def_three_C}
    &\nonumber \mathcal{C}^1_{n,g,\mathbf{p}} :=\somme{j=1}{k}{\mathcal{A}_{g,n,\mathbf{p},j} },\\
    &\mathcal{C}^2_{n,g,\mathbf{p}} :=\somme{p=2}{3g-3+n}{\bigg(m_{p+1}-m_1 m_p}\bigg) \alpha_{n,g,(p,\mathbf{p})} +m_2\alpha_{n,g,(1,\mathbf{p})},\\ & \mathcal{C}^3_{n,g,\mathbf{p}} :=-m_1^2\alpha_{n,g,(1,\mathbf{p})} -m_1 (2g-2+n) \alpha_{n,g,\mathbf{p}}.\nonumber 
    &
\end{align}
Thus, \eqref{recursion_polynomial_with_derivative} can be rewritten
\begin{align*}
    \alpha_{n+1,g,\mathbf{p}} &= \mathcal{C}^1_{n,g,\mathbf{p}} +\mathcal{C}^2_{n,g,\mathbf{p}} +\mathcal{C}^3_{n,g,\mathbf{p}} .
\end{align*}
We recall that \eqref{few_equations_phi} and \eqref{def_phi} give
\begin{align*}
    &\varphi(n,g,(1,\mathbf{p})) \underset{g\to +\infty}{\sim} -\bigg(\frac{-M_1}{M_0}\bigg)^{-1}(3g)\varphi(n,g,\mathbf{p}),\\
    & \varphi(n+1,g,\mathbf{p})=\bigg(\frac{-M_1}{M_0}\bigg)(5g)\varphi(n,g,\mathbf{p}).
\end{align*}
Using these equations and by induction we can write
\begin{align*}
    -\bigg(\frac{-M_1}{M_0}\bigg)^2&\alpha_{n,g,(1,\mathbf{p})}(\mathbf{M}) \underset{g \to +\infty}{\sim} -\bigg(\frac{-M_1}{M_0}\bigg)^2 \varphi(n,g,(1,\mathbf{p})) \sim \bigg(\frac{-M_1}{M_0}\bigg)(3g)\varphi(n,g,\mathbf{p}),\\
    &  \bigg(\frac{-M_1}{M_0}\bigg) (2g-2+n) \alpha_{n,g,\mathbf{p}}(\mathbf{M}) \sim \bigg(\frac{-M_1}{M_0}\bigg)(2g)\varphi(n,g,\mathbf{p}).
\end{align*}
Thus, we have 
\begin{align}\label{asymptotic_dominant_term}
    \mathcal{C}^3_{n,g,\mathbf{p}}(\mathbf{M}) \sim -\bigg(\frac{-M_1}{M_0}\bigg)(5g)\varphi(n,g,\mathbf{p}) = \varphi(n+1,g,\mathbf{p}).
\end{align}
To conclude we only have to show that $\mathcal{C}^1_{n,g,\mathbf{p}}(\mathbf{M})  = o\bigg(\varphi(n+1,g,\mathbf{p})\bigg)$ and $\mathcal{C}^2_{n,g,\mathbf{p}}(\mathbf{M})  = o\bigg(\varphi(n+1,g,\mathbf{p})\bigg)$.  For $\mathcal{C}^1_{n,g,\mathbf{p}}(\mathbf{M})$, we have the bound
\begin{align*}
    |\mathcal{C}^1_{n,g,\mathbf{p}}(\mathbf{M})| \le k \sup_{\substack{ 1 \le j \le k}}\bigg|\mathcal{A}_{g,n,\mathbf{p},j}(\mathbf{M})\bigg|.
\end{align*}
Using Proposition~\ref{uniform_bound_derivative_and_A} we find
\begin{align*}
    |\mathcal{C}^1_{n,g,\mathbf{p}}(\mathbf{M})| \le C^{'}_n k\bigg(\frac{-M_1}{M_0}\bigg)^{n+1-|\mathbf{p}|}g^{n+k}A^k\varphi(g) \underset{\eqref{asymptotic_phi_n_p_fixed}}{=} o\bigg(\varphi(n+1,g,\mathbf{p})\bigg),
\end{align*}
where the last equality uses the fact that here $n$ and $\mathbf{p}$ are fixed. We let the reader verify that using the first bound of Proposition~\ref{uniform_bound_derivative_and_A} and similar arguments we have $\mathcal{C}^2_{n,g,\mathbf{p}}(\mathbf{M})= o\bigg(\varphi(n+1,g,\mathbf{p})\bigg)$. This concludes the proof.
\end{proof}
We recall that we have
\begin{align*}    \alpha_{n,g,\mathbf{p},\mathbf{q}} = [L_1^{2q_1}\cdots L_n^{2q_n}]\frac{\partial \mathcal{P}_{g,n}}{\partial m_{\mathbf{p}}}\bigg(\sqrt{-3m_1}\mathbf{L},\mathbf{m}\bigg).
\end{align*}
The next proposition gives an asymptotic for the coefficients $\alpha_{n,g,\mathbf{p},\mathbf{q}}(\mathbf{M})$. 
\begin{proposition}\label{asymptotics_of_coefficients}
    For any $\mu_g \underset{g \to +\infty}{\rightarrow}\mu_c$ such that $\mu_c - \mu_g = o(g^{-2})$ and any $p_1,\cdots,p_k \ge 1$, $n\ge0$ and $q_1,\cdots,q_n \ge 0$, we have
    \begin{align}\label{limit_coefficient}        \alpha_{n,g,\mathbf{p},\mathbf{q}}(\mathbf{M})\underset{g \to +\infty}{\sim} \varphi(n,g,\mathbf{p}) \prod_{i=1}^{n}\frac{1}{(2q_i+1)!}.
    \end{align}
Furthermore, for $g$ large enough we have
\begin{align}\label{bound_coefficient}
        \forall \mathbf{p}\in (\mathbb{N}^{*})^{k},\hspace{0.2cm}\forall \mathbf{q}\in \mathbb{N}^n,\hspace{0.3cm}\big|\alpha_{n,g,\mathbf{p},\mathbf{q}}(\mathbf{M}) \big|\le C_n  \bigg(-\frac{M_1}{M_0}\bigg)^{n-|\mathbf{p}|}g^{n+k}A^k\varphi(g)\prod_{i=1}^{n}\frac{1}{q_i!},
    \end{align}
    where $C_n > 0$ only depends on $n$ and $A > 0$ is an absolute constant. 
\end{proposition}
\begin{proof}
    We prove both statements by induction on $n \in \mathbb{N}$.\\
\textbf{Initial case:}\\
The case $n = 0$ is a direct consequence of Proposition~\ref{uniform_bound_derivative_and_A} and Proposition~\ref{Estimate_derivative_Pg0}.\\
\textbf{Induction case:}\\
Suppose that the result is true for $n\ge 0$. Fix $q_1,\cdots, q_n, q_{n+1} \ge 0$ and $p_1,\dots,p_k \ge 1$. We distinguish different cases according to $q_1$. Since the polynomial $\mathcal{P}_{g,n}(\mathbf{L},\mathbf{m})$ is symmetric in $\mathbf{L}$, we can assume without loss of generality that $q_1 \ge \cdots \ge q_{n+1}$.\\
\fbox{$\mathbf{q_1 \ge 2:}$}\\
\noindent By Theorem~\ref{polynomial_recursion}, we write 
\begin{align}\label{term_q1_greater_2}
     & \alpha_{n+1,g,\mathbf{p},\mathbf{q}}(\mathbf{M})=-\frac{1}{2^{q_1}q_1!}\bigg(-\frac{M_1}{3M_0}\bigg)^{q_1} \alpha_{n,g,(q_1-1,\mathbf{p}),\mathbf{q}_{\ge 2}}(\mathbf{M}).
\end{align}
Note that this equality still makes sense when $q_1 \ge 3g-3+n+2$ since the two terms in the quality are equal to $0$. Working with $q_1,\cdots,q_{n+1}$ and $p_1,\cdots,p_{k}$ fixed, applying the induction result for $n$ we obtain that the last display is equivalent as $g\to \infty$ to
\begin{align}\label{intermediate_term}
    &-\frac{1}{2^{q_1}q_1!}\bigg(-\frac{M_1}{3M_0}\bigg)^{q_1} \varphi\big(n,g,(q_1-1,\mathbf{p})\big)\prod_{i=2}^{n+1}\frac{1}{(2q_i+1)!}.
\end{align}
We recall that \eqref{few_equations_phi2} gives
\begin{align*}
    -\frac{1}{2^{q_1}q_1!}\bigg(-\frac{M_1}{3M_0}\bigg)^{q_1} \varphi\bigg(n,g,(q_1-1,\mathbf{p})\bigg) \underset{g \to +\infty}{\sim} \frac{1}{(2q_1+1)!}\varphi(n+1,g,\mathbf{p}).
\end{align*}
This concludes the first part of the proposition in the case $q_1\ge 2$. To bound $\big|\alpha_{n+1,g,\mathbf{p},\mathbf{q}}(\mathbf{M}) \big|$,
we simply use the induction case applied to the right-hand side of \eqref{term_q1_greater_2}, making sure to take $C_{n+1} \ge A C_n$.\\
\fbox{$\mathbf{q_1 = 1:}$}\\
 Using Theorem~\ref{polynomial_recursion}, we write $\alpha_{n+1,g,\mathbf{p},\mathbf{q}}(\mathbf{M})$ as
\begin{align}\label{the_term_n+1}
-\frac{M_1}{6M_0}\bigg[&\somme{p=1}{3g-3+n}{\frac{M_p}{M_0}}\alpha_{n,g,(p,\mathbf{p}),\mathbf{q}_{\ge 2}}(\mathbf{M}) +\somme{\substack{i=1}}{k}{\alpha_{n,g,(p_i,\mathbf{p}_{\widehat{\{i\}}}),\mathbf{q}_{\ge 2}}}(\mathbf{M})  + (2g-2+n)\alpha_{n,g,\mathbf{p},\mathbf{q}_{\ge 2}}(\mathbf{M})\bigg].
\end{align}
Observing that $\alpha_{n,g,(p_i,\mathbf{p}_{\widehat{\{i\}}}),\mathbf{q}_{\ge 2}} = \alpha_{n,g,\mathbf{p},\mathbf{q}_{\ge 2}} $, this becomes
\begin{align}\label{the_term_n+1}
\alpha_{n+1,g,\mathbf{p},\mathbf{q}}(\mathbf{M})=-\frac{M_1}{6M_0}\bigg[&\somme{p=1}{3g-3+n}{\frac{M_p}{M_0}}\alpha_{n,g,(p,\mathbf{p}),\mathbf{q}_{\ge 2}}(\mathbf{M})   + (2g-2+n+k)\alpha_{n,g,\mathbf{p},\mathbf{q}_{\ge 2}}(\mathbf{M})\bigg].
\end{align}
Using the same arguments as in the case $q_1 \ge 2$, we find
\begin{align}\label{first_term}
    &-\frac{M_1}{6M_0}\bigg[\frac{M_1}{M_0}\alpha_{n,g,(1,\mathbf{p}),\mathbf{q}_{\ge 2}}(\mathbf{M})   + (2g-2+n+k)\alpha_{n,g,\mathbf{p},\mathbf{q}_{\ge 2}}(\mathbf{M})\bigg]\\ &\nonumber \sim \frac{1}{6}\varphi(n+1,g,\mathbf{p})\prod_{i=2}^{n+1}\frac{1}{(2q_i+1)!}.
\end{align}
For $p \ge 2$, combining the induction for $n$ and \eqref{bound_Mk}, we obtain 
\begin{align}\label{only_first_term_matters}
    &\bigg|-\frac{M_1M_p}{6M_0^2}\alpha_{n,g,(p,\mathbf{p}),\mathbf{q}_{\ge 2}}(\mathbf{M})\bigg| \le  C_n  \bigg(-\frac{M_1}{M_0}\bigg)^{n-|\mathbf{p}|-p+2}g^{n+1+k}A^k\varphi(g)\prod_{i=1}^{n+1}\frac{1}{q_i!},
\end{align}
where $C_n > 0$ only depends on $n$ and $A > 0$ is an absolute constant. Using \eqref{asymptotic_phi_n_p_fixed}, as $g \to +\infty$ we have $\bigg(\displaystyle -\frac{M_1}{M_0}\bigg)^{n-|\mathbf{p}|}g^{n+1+k}\varphi(g) = o\bigg(\varphi(n+1,g,\mathbf{p})\bigg)$. Then using \eqref{bound_Mk} and the fact that $M_0 \underset{g \to +\infty}{\rightarrow}0$, we deduce that for $g$ large enough we have
\begin{align}\label{bound_for_p_greater_2}
    \bigg|\somme{p=2}{3g-3+n}{}-\frac{M_1M_p}{6M_0^2}\alpha_{n,g,(p,\mathbf{p}),\mathbf{q}_{\ge 2}}(\mathbf{M})\bigg| &\le  C^{'}_n  \cdot \bigg(-\frac{M_1}{M_0}\bigg)^{n-|\mathbf{p}|}g^{n+1+k}A^k\varphi(g) \prod_{i=1}^{n+1}\frac{1}{q_i!}\\&= o\bigg(\varphi(n+1,g,\mathbf{p})\bigg).\nonumber
\end{align}
Combining \eqref{the_term_n+1}, \eqref{first_term} and \eqref{bound_for_p_greater_2} we deduce \eqref{limit_coefficient} for $n+1$ using the induction result for $n$. The bound \eqref{bound_coefficient} is obtained the same way by applying the induction to \eqref{the_term_n+1}, using \eqref{bound_for_p_greater_2}.\\
\fbox{$\mathbf{q_1 =0:}$}\\
If $q_1 = 0$, then $q_1 = \cdots = q_{n+1} = 0$. In that case we have 
\begin{align*}
   \displaystyle \alpha_{n+1,g,\mathbf{p},\mathbf{q}}(\mathbf{M}) = \alpha_{n+1,g,\mathbf{p}}(\mathbf{M}).
\end{align*}
We conclude using Proposition~\ref{uniform_bound_derivative_and_A} and Proposition~\ref{Estimate_derivative_Pg0}.
\end{proof}

Finally, we conclude the proof of Proposition~\ref{estimate_boundary}.
\begin{proof}[Proof of Proposition~\ref{estimate_boundary}]\label{proof_estimate_boundary}
The proof of the asymptotics is a direct consequence of Theorem~\ref{polynomial_recursion} and Proposition~\ref{asymptotics_of_coefficients}. Indeed, we can write $T_{g,n}\bigg(\bigg(-\frac{M_1}{3M_0}\bigg)^{\frac{1}{2}}\mathbf{L}\bigg) = \displaystyle \frac{1}{M_0^{2g-2+n}}\mathcal{P}_{g,n}\bigg(\bigg(-\frac{M_1}{3M_0}\bigg)^{\frac{1}{2}}\mathbf{L},\mathbf{M}\bigg)$. Then we can write
\begin{align*}
    \mathcal{P}_{g,n}\bigg(\bigg(-\frac{M_1}{3M_0}\bigg)^{\frac{1}{2}}\mathbf{L},\mathbf{M}\bigg) = \displaystyle \somme{\mathbf{q}\in \mathbb{N}^{n}}{}{\alpha_{n,g,\emptyset,\mathbf{q}}}L_1^{2q_1}\cdots L_{n}^{2q_n}.
\end{align*}
Applying the inequality of Proposition~\ref{asymptotics_of_coefficients} and using the third equation of \eqref{few_equations_phi} with $n$ and $p=  \emptyset$ we have for $g$ large enough
\begin{align*}    \forall \mathbf{q}\in \mathbb{N}^{n},\hspace{0.5cm} |\alpha_{n,g,\emptyset,\mathbf{q}}| \le C_n |\alpha_{n,g}| \prod_{i=1}^{n}\frac{1}{q_i!} .
\end{align*}
Moreover for $\mathbf{q}$ fixed, using Proposition~\ref{asymptotics_of_coefficients} we find
\begin{align*}   \alpha_{n,g,\emptyset,\mathbf{q}} \sim  \alpha_{n,g}\prod_{i=1}^{n}\frac{1}{(2q_i+1)!},
\end{align*}
which verifies the claimed asymptotic formula.

To prove the second part of the statement, let us reason by contradiction. Suppose that for any $g_0 \ge 0$ there exist $g > g' \geq g_0$ and $\mathbf{L}\in K$ such that 
\begin{align*}
    T_{g',n}\bigg(\bigg(-\frac{M_1}{3M_0}\bigg)^{\frac{1}{2}}\mathbf{L},\mu_{g}\bigg) > 2e^{\somme{i=1}{n}{L_i}}T_{g',n}(\mu_{g}).
\end{align*}
Choosing an explicit dependence of such values $g = \psi(g_0)$, $g' = \psi'(g_0)$ on $g_0$, we consider the sequence $(v_k)$ given by the recursion relation $v_0 = 0$ and $v_{k+1} = \psi(v_k)$.
For $k \ge 1$, we define $g_k = \psi'(v_{k-1})$, such that in particular $g_k\to \infty$ as $k\to\infty$.
It follows that there exists a sequence $(\mathbf{L}_k)$ in $K$, such that the inequality
\begin{align*}
    T_{g_k,n}\bigg(\bigg(-\frac{M_1}{3M_0}\bigg)^{\frac{1}{2}}\mathbf{L}_k,\mu_{v_k}\bigg) > 2e^{\somme{i=1}{n}{L_i}}T_{g_k,n}(\mu_{v_k})
\end{align*}
holds for all $k\geq 1$.
Since $g_k\le v_k$ and therefore $\mu_c - \mu_{v_k} = o(v_k^{-2}) = o(g_k^{-2})$, we obtain a contradiction with the first part of the proposition.
\end{proof}

We finish this section with an estimate that we will need later (in the proof of Proposition~\ref{number_of_simple_tight_closed_geodesics}).

\begin{proposition}\label{bound_separating_multicurve}
   For any $\mu_g \underset{g \to +\infty}{\rightarrow}\mu_c$ such that $\mu_c- \mu_g = o(g^{-2})$, any $r >0$ and $1 < q \le r+1$, we have for $g$ large enough
    \begin{align*}
       \frac{1}{M_0(\mu_g)^{r} T_{g}(\mu_g)}\cdot \somme{(g_1,n_1),\ldots,(g_q,n_q)}{}{}\displaystyle \prod_{i=1}^{q}T_{g_i,n_i}(\mu_g) \le  \frac{C_r}{g^{q-1}},
    \end{align*}
    where the sum is taken over all sequences $(g_1,n_1),\ldots,(g_q,n_q)$ such that $2g_i+n_i\ge 3$ and $\somme{i=1}{q}{g_i} = g + q -r -1 $ and $\somme{i=1}{q}{n_i} = 2r$. The constant $C_r > 0$ only depends on $r$ and the sequence $(\mu_g)$.
\end{proposition}
\begin{proof}
To lighten the notation in the proof, we do not specify the dependence on $\mu$ since it is always $\mu = \mu_g$.
By Theorem~\ref{polynomial_recursion}, we have
\begin{align*}
    T_{g,n} = \frac{1}{M_0^{2g-2+n}}\mathcal{P}_{g,n}(\mathbf{M})=\frac{1}{M_0^{2g-2+n}}\alpha_{n,g}(\mathbf{M}).
\end{align*}
Since $\somme{i=1}{q}{2g_i-2+n_i} = 2g-2$ we only need to control 
\begin{align*}
      \somme{(g_1,n_1),\ldots,(g_q,n_q)}{}{}\frac{ \prod_{i=1}^{q}\alpha_{n_i,g_i}(\mathbf{M})}{M_0^{r} \alpha_{g}(\mathbf{M})}.
\end{align*}
Reasoning by contradiction (see proof~\ref{proof_estimate_boundary} for a similar proof) and using Proposition~\ref{asymptotics_of_coefficients}, we let the reader verify that there exists $g_0$ large enough such that for any $g \ge g_0$ and any $g'\in \{g_0,\cdots,g\}$ we have 
\begin{align*}
     \alpha_{n,g'}(\mathbf{M}) \le C_n \bigg(-\frac{M_1}{M_0}\bigg)^{n}(g')^{n} \varphi(g') = C_n \bigg(-\frac{M_1}{M_0}\bigg)^{3g'-3+n}(g')^{n}\frac{\langle \tau_2^{3g'-3}\rangle_{g^{'}}}{(3g'-3)!},
\end{align*}
where $C_n$ only depends on $n$. Note that $g_0$ depends on the sequence $(\mu_g)$ chosen.\\
Moreover, using Proposition~\ref{crude_bound}, for $g$ large enough we can write
\begin{align*}
   \forall g' \in \{0,\cdots,g_0-1\},\hspace{0.1cm}\forall n' \in \{0,\cdots,r\},\hspace{0.3cm} \alpha_{n',g'}(\mathbf{M}) = \mathcal{T}_{g',n'}(\mathbf{M}) \le C_{n',g'} \bigg(-\frac{M_1}{M_0}\bigg)^{3g-3+n},
\end{align*}
where $C_{n',g'} > 0$ is a constant that only depends on $n'$ and $g'$.\\
Let $\displaystyle C : = \max\{\max_{1\leq n\leq r} C_n^r, \max_{1 \le n \le r,\,0\le g'\le g_0}C_{n,g'}^r\}$. We may combine these bounds so that for $g$ large enough we have
\begin{align*}
     \frac{\displaystyle \prod_{i=1}^{q} \alpha_{n_i,g_i}(\mathbf{M}) }{\displaystyle \prod_{i=1}^{q}\bigg(-\frac{M_1}{M_0}\bigg)^{3g_i-3+n_i} \displaystyle \prod_{\substack{i=1,\cdots,q\\g_i \ge g_0}}\frac{\langle \tau_2^{3g_{i}-3}\rangle_{g_i}}{(3g_i-3)!}g_{i}^{n_i}} \le C,
\end{align*}
where $C > 0$ is a constant that only depend on $r$ and the sequence $(\mu_g)$.\\
 Using Proposition~\ref{split_sum} and the fact that $\somme{i=1}{q}{3g_i - 3 +n_i} = 3g-3-r$, we deduce that it remains to bound 
\begin{align*}
    \somme{(g_1,n_1),\ldots,(g_q,n_q)}{}{}\frac{\displaystyle \prod_{\substack{i=1,\cdots,q\\g_i \ge g_0}}\frac{\langle \tau_2^{3g_{i}-3}\rangle_{g_i}}{(3g_i-3)!}g_i^{n_i}}{\displaystyle \frac{\langle \tau_2^{3g-3}\rangle_{g}}{(3g-3)!}},
\end{align*}
Using \eqref{intersection_numbers}, we only have to bound
\begin{align*}
    \somme{(g_1,n_1),\ldots,(g_q,n_q)}{}{}\frac{\displaystyle \prod_{\substack{i=1,\cdots,q\\g_i \ge g_0}}\frac{(g_i-1)!^2 }{(5g_i-5)(5g_i-3)}g_i^{n_i}}{\displaystyle \frac{(g-1)!^{2}}{(5g-5)(5g-3)}},
\end{align*}
where we have used the fact that the terms corresponding to $g_i \le g_0$ can be treated as constants that depend on $g_0$ only. Using the Stirling formula, there exists a constant $C > 0$ depending on $r$ such that the last sum is bounded by 
\begin{align*}
    C\somme{(g_1,n_1),\ldots,(g_q,n_q)}{}{}\frac{\prod_{\substack{i=1,\cdots,q\\g_i \ge g_0}}\displaystyle (g_i-1)^{2g_i-3+n_i}}{(g-1)^{2g-3}}.
\end{align*}
We conclude using the bound 
\begin{align*}
    \somme{(g_1,n_1),\ldots,(g_q,n_q)}{}{}\frac{\prod_{\substack{i=1,\cdots,q\\g_i \ge g_0}}\displaystyle (g_i-1)^{2g_i-3+n_i}}{(g-1)^{2g-3}} \le C \frac{1}{g^{q-1}},
\end{align*}
where $C_r > 0$ only depends on $r$. This concludes the proof.
\end{proof}

\subsection{Study of the number of cusps}
This subsection is dedicated to estimates on the distribution of the number $\mathcal{N}_{g,\mu}$ of cusps in the $\mu$-Boltzmann hyperbolic surface of genus $g$, as defined in the introduction. 
More precisely, in the regime $\mu_c-\mu_g = o(g^{-2})$ we give an asymptotic for $\mathbb{E}[\mathcal{N}_{g,\mu_g}]$ and show a concentration phenomenon using a second moment method.\\

\begin{proposition}\label{number_of_cusps_mean_and_concentration}
     For any $\mu_g \underset{g \to +\infty}{\rightarrow}\mu_c$ such that $\mu_c - \mu_g = o(g^{-2})$ and any $r \ge 0$, we have 
    \begin{align*}
        \mathbb{E}[\mathcal{N}_{g,\mu_g}(\mathcal{N}_{g,\mu_g}-1)\cdots (\mathcal{N}_{g,\mu_g}-r)] &\underset{g \to +\infty}{\sim} \bigg(\frac{5g \mu_c}{2(\mu_c - \mu_g)}\bigg)^{r+1}.
    \end{align*}
\end{proposition}
\begin{proof}
 We write
    \begin{align}
        \mathbb{E}[\mathcal{N}_{g,\mu_g}(\mathcal{N}_{g,\mu_g}-1)\cdots (\mathcal{N}_{g,\mu_g}-r)] &= T_{g}(\mu_g)^{-1}\somme{n=0}{+\infty}{n(n-1)\cdots (n-r)\frac{\mu_g^{n}}{n!}V_{g,n}}\nonumber\\
        &=\mu_g^{r+1}\frac{T_{g,r+1}(\mu_g)}{T_{g}(\mu_g)}.\label{number_of_cusps_formula}
    \end{align}
Using Theorem~\ref{polynomial_recursion}, Proposition~\ref{Estimate_derivative_Pg0} and \eqref{def_phi} we deduce that
\begin{align*}
    \frac{T_{g,r+1}(\mu_g)}{T_{g}(\mu_g)} = (-M_1(\mu_g))^{r+1}M_0(\mu_g)^{-2r-2}(5g)^{r+1}\bigg(1+o(1)\bigg)
\end{align*}
Using Proposition~\ref{asymptotics_M1_M0_R} we obtain 
\begin{align*}
    \mathbb{E}[\mathcal{N}_{g,\mu_g}(\mathcal{N}_{g,\mu_g}-1)\cdots (\mathcal{N}_{g,\mu_g}-r)] &= \bigg(\frac{5 g\mu_c}{2(\mu_c - \mu_g)}\bigg)^{r+1}\bigg(1+o(1)\bigg).
\end{align*}
\end{proof}
It follows that we have the following concentration phenomenon for the variables $\mathcal{N}_{g,\mu_g}$.
\begin{corollary}\label{Concentration_result}
      For any $\mu_g \underset{g \to +\infty}{\rightarrow}\mu_c$ such that $\mu_c - \mu_g = o(g^{-2})$, any $r \ge 0$ and any $\varepsilon > 0$, we have
     \begin{align*}
         \Pf\bigg(\big|\mathcal{N}_{g,\mu_g} -\mathbb{E}(\mathcal{N}_{g,\mu_g})\big| \ge \varepsilon \mathbb{E}(\mathcal{N}_{g,\mu_g})\bigg) \underset{g \to +\infty}{\rightarrow}0.
     \end{align*}
\end{corollary}
\begin{proof}
With the help of Proposition~\ref{number_of_cusps_mean_and_concentration} we find that
\begin{align}\label{variance_computation}
    \frac{\operatorname{Var}(\mathcal{N}_{g,\mu_g})}{\mathbb{E}(\mathcal{N}_{g,\mu_g})^2} &= \frac{\mathbb{E}(\mathcal{N}_{g,\mu_g}(\mathcal{N}_{g,\mu_g}-1))+\mathbb{E}(\mathcal{N}_{g,\mu_g})}{\mathbb{E}(\mathcal{N}_{g,\mu_g})^2} -1 \underset{g \to +\infty}{\rightarrow}0.
\end{align}
We conclude with the Bienaymé-Chebychev inequality.
\end{proof}
\section{Tight length spectrum}\label{proof_main_theorem}
This section is dedicated to proving the main Theorem~\ref{main_theorem}. 
Recall its equivalent formulation in terms of the random variables $N_{g,\mu,a,b}^{\mathrm{tight}}$ counting the primitive tight closed geodesics with length in $[\alpha_1^{-1}(\mu_c-\mu)^{-\frac{1}{4}}a,\alpha_1^{-1}(\mu_c-\mu)^{-\frac{1}{4}}b]$ in a random hyperbolic surface $X$ chosen under $\Pf^{\mathrm{WP}}_{g,\mu}$, with the constant $\alpha_1 = \sqrt{\frac{6}{\pi}\sqrt{\frac{2j_0}{J_1(j_0)}}}$ as in Theorem~\ref{main_theorem}. Observe from Proposition~\ref{asymptotics_M1_M0_R} that we have
\begin{align}\label{change_coefficient}
    \bigg(-\frac{M_1}{12M_0}\bigg)^{\frac{1}{2}} \underset{\mu \to \mu_c}{\sim} \alpha_1^{-1}(\mu_c-\mu)^{-\frac{1}{4}}.
\end{align}
For the sake of simplifying the notation we assume for the rest of this section that $N_{g,\mu,a,b}^{\mathrm{tight}}$ denotes the number of primitive tight closed geodesics with length in 
\begin{align*}\bigg[ \displaystyle \bigg(-\frac{M_1}{12M_0}\bigg)^{\frac12} a, \displaystyle \bigg(-\frac{M_1}{12M_0}\bigg)^{\frac12} b\bigg].
\end{align*}
Indeed, if one proves the convergence in distribution with this definition for $N_{g,\mu,a,b}^{\mathrm{tight}}$, the statement of Theorem~\ref{main_theorem} follows easily. For $X$ a random variable taking values in $\mathbb{N}$ we introduce the notation
\begin{align}\label{moment_with_increment}
    (X)_k = X(X-1)\cdots(X-k+1), \hspace{0.5cm} k \ge 0
\end{align}

\subsection{Moments}

Let us follow the same strategy as in \cite{Mirzakhani_petri_2019}. It amounts to proving that the joint factorial moments $\mathbb{E}\bigg[(N^{\mathrm{tight}}_{g,\mu_g,a_1,b_1})_{r_1}\cdots(N^{\mathrm{tight}}_{g,\mu_g,a_n,b_n})_{r_n}\bigg]$ converge as $g \to +\infty$ to $\lambda_{a_1,b_1}^{r_1}\cdots\lambda_{a_n,b_n}^{r_n}$. To do so, we interpret 
\begin{align*}
(N^{\mathrm{tight}}_{g,\mu_g,a_1,b_1})_{r_1}\cdots(N^{\mathrm{tight}}_{g,\mu_g,a_n,b_n})_{r_n}
\end{align*}
as the number of $n$-tuples where the $i^{\mathrm{th}}$ element is an $r_i$-tuple made of distinct primitive tight closed geodesics with length in $\displaystyle \bigg[\bigg(-\frac{M_1}{12M_0}\bigg)^{\frac{1}{2}}a_i,\bigg(-\frac{M_1}{12M_0}\bigg)^{\frac{1}{2}}b_i\bigg]$. We write
\begin{align*}
     (N^{\mathrm{tight}}_{g,\mu_g,a_1,b_1})_{r_1}\cdots(N^{\mathrm{tight}}_{g,\mu_g,a_n,b_n})_{r_n} &= \hat{N}_{g,\mu_g,a_1,b_1,r_1,\dots,a_n,b_n,r_n} \\&+ N^{\times}_{g,\mu_g,a_1,b_1,r_1,\dots,a_n,b_n,r_n},
\end{align*} where the term $\hat{N}_{g,\mu_g,a_1,b_1,r_1,\dots,a_n,b_n,r_n}$ counts the number of tuples where all geodesics are simple and disjoint and the term $N^{\times}_{g,\mu_g,a_1,b_1,r_1,\dots,a_n,b_n,r_n}$ counts the tuples in which either there is a non-simple geodesic or at least two distinct geodesics intersect.

The expectation of $\hat{N}_{g,\mu_g,a_1,b_1,r_1\dots,a_n,b_n,r_n}$ can be easily computed using the integration formula \eqref{integration_formula_version_generating_function} and the estimates for tight Weil-Petersson volumes obtained in Section~\ref{estimates_tight_volumes}.

\begin{proposition}\label{number_of_simple_tight_closed_geodesics}
     For any $\mu_g \underset{g \to +\infty}{\rightarrow}\mu_c$ such that $\mu_c - \mu_g = o(g^{-2})$, any disjoint compact intervals $([a_i,b_i])_{1\le i \le n}$ and any $r_1,\cdots,r_n \ge 1$ we have
    \begin{align*}
        \mathbb{E}\bigg[\hat{N}_{g,\mu_g,a_1,b_1,r_1\dots,a_n,b_n,r_n}\bigg] \underset{g \to +\infty}{\overset{}{\longrightarrow}} \lambda_{a_1,b_1}^{r_1}\cdots\lambda_{a_n,b_n}^{r_n}.
    \end{align*}
\end{proposition}
\begin{proof}
We define $A = \prod_{i=1}^{n}[a_i,b_i]^{r_i}$ and $r=r_1+\cdots+r_n$. Using Proposition~\ref{prop:integrationformula} in its generating function formulation given in \eqref{integration_formula_version_generating_function} over all MCG$(\Sigma_g)$-orbits of ordered lists $(\Gamma_1,\cdots,\Gamma_n)$ where $\Gamma_i = (\gamma_{i,1},\cdots,\gamma_{i,r_i})$ is an ordered list of disjoint simple closed curves and for $i \neq j$ we have $[\Gamma_i]\cap [\Gamma_j] = \emptyset$, we can write
\begin{align*}
    \mathbb{E}\bigg[\hat{N}_{g,\mu_g,a_1,b_1,r_1\dots,a_n,b_n,r_n}\bigg] = T_{g}(\mu_g)^{-1}\somme{[\Gamma]}{}{C_{\Gamma}\inte{(-\frac{M_1}{12M_0})^{\frac{1}{2}}A}{}{T_{g}(\Gamma,x,\mu_g)x_1 \cdots x_r}{x_1\cdots dx_r}}.
\end{align*}
In the sum, we distinguish the term $[\Gamma_0]$ which corresponds to the non-separating case (see Figure~\ref{non_separating_and_separating} below). This term is written
\begin{align*}
    T_{g}(\mu_g)^{-1}2^{-r}\bigg(-\frac{M_1}{12M_0}\bigg)^{r}\inte{A}{}{T_{g-r,2r}\bigg(\bigg(-\frac{M_1}{12M_0}\bigg)^{\frac12}\bm{x},\mu_g\bigg)x_1\cdots x_r}{x_1\cdots dx_r}.
\end{align*}
Using Proposition~\ref{estimate_boundary} and the compactness of $A$, as $g \to +\infty$ the last term is equivalent to
\begin{align*}
   2^{-r}\bigg(-\frac{M_1}{12M_0}\bigg)^{r}\frac{T_{g-r,2r}(\mu_g)}{T_{g}(\mu_g)}\prod_{i=1}^{n}\bigg(\inte{a_i}{b_i}{\frac{\sinh(x_i/2)^2}{x_i^2/4}x_i}{x_i}\bigg)^{r_i} =  \bigg(-\frac{M_1}{12M_0}\bigg)^{r}\frac{\mathcal{T}_{g-r,2r}(\mathbf{M})}{\mathcal{T}_{g}(\mathbf{M})}\prod_{i=1}^{n}\lambda_{a_i,b_i}^{r_i},
\end{align*}
where the equality follows from Theorem~\ref{polynomial_recursion}. Combining Proposition~\ref{Estimate_derivative_Pg0} and \eqref{intersection_numbers}, we deduce
\begin{align*}
    \bigg(-\frac{M_1}{12M_0}\bigg)^{r}\frac{\mathcal{P}_{g-r,2r}(\mathbf{M})}{\mathcal{P}_{g}(\mathbf{M})} \underset{g \to +\infty}{\rightarrow}1.
\end{align*}
\begin{figure}[H]
    \centering
    \includegraphics[scale=0.35]{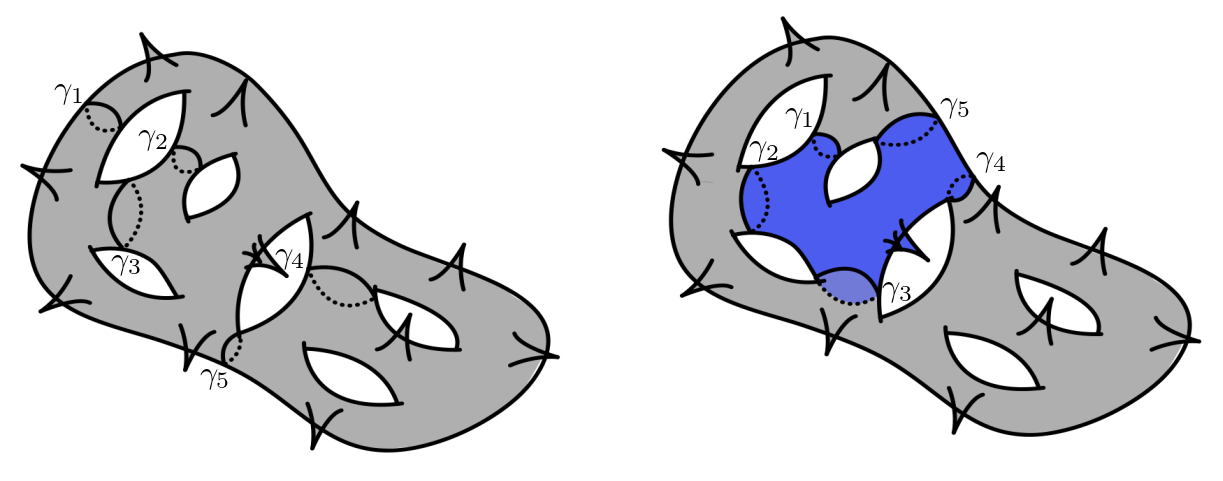}
    \caption{For a multicurve $\Gamma = (\gamma_1,\gamma_2,\gamma_3,\gamma_4,\gamma_5)$, the two different topological cases we distinguish : on the left the non-separating case and on the right the separating case. On the right, we have $q(\Gamma)=2$.}
    \label{non_separating_and_separating}
\end{figure}
It remains to show that the contribution of separating multicurves $[\Gamma]$ is negligible. Let us write:
\begin{align}\label{I_sep}
    I_{sep} = T_{g}(\mu_g)^{-1}\somme{[\Gamma] \ne [\Gamma_0]}{}{C_{\Gamma}\inte{(-\frac{M_1}{12M_0})^{\frac{1}{2}}A}{}{T_{g}(\Gamma,x,\mu_g)x_1 \cdots x_r}{x_1\cdots dx_r}}
\end{align}
Let us order the orbits $[\Gamma]$ by how many connected components $\Sigma_{g} \backslash \Gamma$ is made of, we call $q(\Gamma)$ this quantity (see Figure~\ref{non_separating_and_separating}). We need to bound
\begin{align}\label{bound_separating_case}
    T_{g}(\mu_g)^{-1}\somme{q=2}{r}{\somme{\substack{[\Gamma]\\q(\Gamma)=q}}{}{\inte{(-\frac{M_1}{12M_0})^{\frac{1}{2}}A}{}{T_{g}(\Gamma,x,\mu_g)x_1 \cdots x_r}{x_1\cdots dx_r}}},
\end{align}
where we have used the fact that $C_{\Gamma}  \le 1$.
For $\Gamma$ fixed, we write $\Sigma_g \backslash \Gamma = \displaystyle \bigsqcup_{i=1}^{q} \Sigma_{g_i,n_i}$, then the integral in the last sum can be rewritten
\begin{align}\label{bound_with_boundary}
     T_{g}(\mu_g)^{-1}\inte{(-\frac{M_1}{12M_0})^{\frac{1}{2}}A}{}{\prod_{i=1}^{q}T_{g_i,n_i}(\bm{x}^{(i)},\mu_g)x_1 \cdots x_{r}}{x_1\cdots \mathrm{d}x_{r}},
\end{align}
where $\bm{x}^{(i)}$ denotes the tuple of coordinates $x_{j}$ of $\bm{x}$ such that $\gamma_{j}$ is a boundary of $\Sigma_{g_i,n_i}$ and $A_i = \prod_{k=1}^{n_i}[a_{i_k},b_{i_k}]$. Let us give a uniform bound for $T_{g_i,n_i}(\bm{x},\mu_g)$.\\ 
By Proposition~\ref{estimate_boundary}, there exists a $g_0 \ge 0$ such that for any $g \ge g_0$, any $g' \in \{g_0,\cdots,g\}$, any $k \in\{0,\cdots,r\}$ and any $\mathbf{L}\in [0,b_n]^k$, we have 
\begin{align*}
    T_{g',k}\bigg(\bigg(-\frac{M_1}{3M_0}\bigg)^{\frac{1}{2}}\mathbf{L},\mu_g\bigg) \le 2e^{rb_n}T_{g',k}(\mu_g),
\end{align*}
The constant $g_0$ depends on $(\mu_g)$, $n$ and $b_n$.  \\
Now, we fix such a $g_0$. Using Proposition~\ref{crude_bound}, for $g$ large enough, for any $g' \in \{0,\cdots,g_0-1\}$, any $k \in\{0,\cdots,r\}$ and any $\mathbf{L}\in [0,b_n]^k$, we have
\begin{align*}
    T_{g',k}\bigg(\bigg(-\frac{M_1}{3M_0}\bigg)^{\frac{1}{2}}\mathbf{L},\mu_g\bigg) \le CT_{g',k}(\mu_g),
\end{align*}
where the constant $C > 0$ depends on $r,b_n$ and $(\mu_g)$. Now we can bound \eqref{bound_with_boundary} by 
\begin{align*}
   C\bigg(-\frac{M_1}{12M_0}\bigg)^{r}\frac{\prod_{i=1}^{q}T_{g_i,n_i}(\mu_g)}{ T_{g}(\mu_g)},
\end{align*}
where $C > 0$ is a positive constant which only depends on $r,b_n$ and $(\mu_g)$.\\
To conclude, observe that the sum \eqref{bound_separating_case} can be bounded by:
\begin{align*}
    I_{\mathrm{sep}} \leq C\somme{q=2}{r}{\somme{(g_i,n_i)}{}{(2r)!!}\displaystyle \frac{\prod_{i=1}^{q}T_{g_i,n_i}(\mu_g)}{M_0^{r}T_{g}(\mu_g)}}
\end{align*}
where $C > 0$ only depends on $r,b_n$ and $(\mu_g)$. The term $\displaystyle (2r)!!$ bounds the number of ways $\displaystyle \bigsqcup_{i=1}^{q} \Sigma_{g_i,n_i}$ can be glued to obtain $\Sigma_g$. We conclude using Proposition~\ref{bound_separating_multicurve}.
\end{proof}

Giving a bound on the expectation of $N^{\times}_{g,\mu_g,a_1,b_1,r_1,\dots,a_n,b_n,r_n}$ requires extra control on the number of tight simple closed geodesics of length smaller than $(-\frac{M_1}{12M_0})^{\frac{1}{2}}b_n$ and on the size of the tight systole. Unfortunately, no such estimate is known. Thus its computation is out of reach.

 For any fixed $L > 0$ and $g \ge 0$, let us introduce the event $\mathcal{A}_{g,L}$ on which there exists a pair of intersecting tight curves of length at most $(-\frac{M_1}{12M_0})^{\frac{1}{2}}L$ or a self-intersecting tight curve of length at most $(-\frac{M_1}{12M_0})^{\frac{1}{2}}L$. On the event $\mathcal{A}_{g,b_n}^{c}$, we have $N^{\times}_{g,\mu_g,a_1,b_1,r_1,\dots,a_n,b_n,r_n} = 0$.

In the next Proposition, we prove that $\mathcal{A}_{g,b_n}^{c}$ occurs with high probability.
\begin{proposition}\label{number_of_closed_geodesics_with_inter_or_non_simple_1}
     For any $\mu_g \underset{g \to +\infty}{\rightarrow}\mu_c$ such that $\mu_c - \mu_g = o(g^{-2})$, any $L\ge 0$, we have 
    \begin{align*}        \mathbb{P}(\mathcal{A}_{g,L}) \underset{g \to +\infty}{\overset{}{\longrightarrow}} 0.
    \end{align*}
\end{proposition}
\begin{proof}
Proposition~\ref{number_of_simple_tight_closed_geodesics} and a first moment method imply that
\begin{align*} 
    \limsup\limits_{g\to +\infty} \Pf(\hat{N}_{g,\mu_g,0,\varepsilon} >0) \le \lambda_{0,\varepsilon}.
\end{align*}
Using the fact that $\lambda_{0,\varepsilon} \underset{\varepsilon \to 0}{\to}0$, it is sufficient to prove that for any $\varepsilon > 0$ we have 
\begin{align*}
\mathbb{P}(\mathcal{A}_{g,L},\hat{N}_{g,\mu_g,0,\varepsilon} =0) \underset{g \to +\infty}{\overset{}{\longrightarrow}} 0.
\end{align*}

Let us fix $\varepsilon > 0$. On the event $\mathcal{A}_{g,L}$, we either fix a pair $\Gamma=\{\gamma_1,\gamma_2\}$ such that $\gamma_1$ and $\gamma_2$ intersect or a singleton $\Gamma=\{\gamma\}$ with $\gamma$ self-intersecting. Let us prove that for $C(\varepsilon,L) = \frac{4L}{\varepsilon}$ there exists a tight multicurve composed of at most $C(\varepsilon,L)$ disjoint curves with total length less than $8L(-\frac{M_1}{12M_0})^{\frac{1}{2}}$ that separates $X$ in at least two non-trivial parts.

We interpret $\Gamma$ as a ribbon graph in $X$ with vertices corresponding to the intersection points and edges to the segments of the curves between the intersection points. 
We denote its genus by $g'$. 
Without loss of generality we will assume $g' < g/2$, because if $g' \geq g/2$ we may replace $\Gamma$ by a connected ribbon subgraph of genus smaller than $g/2$ as follows.
When removing an edge from a ribbon graph of genus $h$, either the ribbon graph stays connected and has genus $h$ or $h-1$, or it disconnects into two ribbon graphs one of which has genus between $h/2$ and $h$.
Successive removal of edges therefore allows us to extract a subgraph of $\Gamma$ of genus $g'$ satisfying $1 \leq g' < g/2$.


Let $\Gamma^{'}$ be a regular neighbourhood of $\Gamma$ (or the subgraph just constructed) whose boundary is composed of $r'$ disjoint simple closed curves $\alpha_1,\cdots,\alpha_{r'}$. 
Writing $X \backslash \Gamma^{'} = \bigsqcup_{i=1}^{p}S_i$, we consider the set of indices $I$ such that $i \in I$ if and only if $S_i$ is a disk with cusps bounded by one of the $\alpha_j$ or a cylinder with cusps bounded by two distinct curves $\alpha_p$ and $\alpha_q$. Then we consider $\Gamma^{*} = \Gamma^{'} \cup \bigsqcup_{i\in I} S_i$ (see Figure~\ref{separating_subsurface}).
This is a subsurface of $X$ of genus $g^* \geq g'$ with $r^* \leq r'$ boundaries $\beta_1,\ldots,\beta_{r^*}$ satisfying the following properties:
\begin{itemize}
    \item The total length of the boundaries $\beta_1,\ldots,\beta_{r^*}$ is at most $4L(-\frac{M_1}{12M_0})^{\frac{1}{2}}$, which follows from $\Gamma$ having total length less at most $2L(-\frac{M_1}{12M_0})^{\frac{1}{2}}$.  
    \item $2(g^*-g') + r^* \leq C(\varepsilon,L)$. To see this, let $J$ be the set of indices $j$ such that $\alpha_j$ separates a disk with cusps from the rest of the surface. In that case $2(g^{*}-g^{'})+r^{*} = r'-|J|$. On $\{\hat{N}_{g,\mu_g,0,\varepsilon} = 0\}$, any closed curve that is not separating a genus $0$ part from the rest of the surface has length at least $(-\frac{M_1}{12M_0})^{\frac{1}{2}}\varepsilon$. Thus, for each $i \in \{1,\cdots,r'\}\backslash J$ the curve $\alpha_i$ has length at least $(-\frac{M_1}{12M_0})^{\frac{1}{2}}\varepsilon$. Since the total length of the $\alpha_i$ is bounded by $4L(-\frac{M_1}{12M_0})^{\frac{1}{2}}$, we deduce $r'-|J| \le \frac{4L}{\varepsilon} = C(\varepsilon,L)$.
    \item If $g > C(\varepsilon,L)$, then $\Gamma^* \neq X$. This follows from our assumption on $g'$ and the previous property, since $g > C(\varepsilon,L)$ implies $g^* \leq g' + C(\varepsilon,L)/2 < g/2 + C(\varepsilon,L)/2 < g$.
    \item $\Gamma^*$ is not homeomorpic to a cylinder with cusps or to a disk with cusps. This is immediate in the case a ribbon subgraph of $\Gamma$ was used, since then $g^* \geq g' \geq 1$ by construction. Otherwise, this follows from the fact that $\Gamma^* \supset \Gamma$ contains either a pair of distinct tight closed geodesics or a self-intersecting tight closed geodesic, while a cylinder with cusps has at most one primitive tight closed geodesic and a disk with cusps has none. 
\end{itemize}
\begin{figure}[H]
    \centering
    \includegraphics[scale = 0.22]{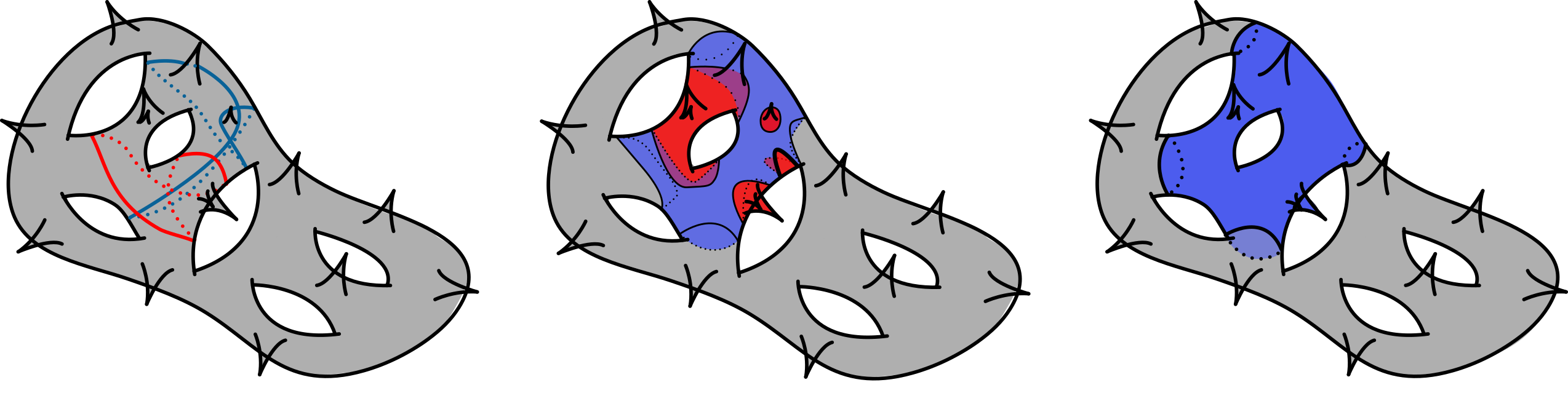}
    \caption{On the left, a multicurve $\Gamma= \{\gamma_1,\gamma_2\}$ with intersections. In the middle, we represented the regular neighbourhood $\Gamma^{'}$ in blue and $\Gamma^{*}$ obtained from $\Gamma^{'}$ by filling the red part. On the right, the blue part corresponds by the subsurface separated when taking the tight multicurve associated to the boundaries of $\Gamma^{*}$ and cutting along it. }
    \label{separating_subsurface}
\end{figure}
Applying Proposition~\ref{tighten_multicurve}, the tight multicurve associated to the boundaries of $\Gamma^{*}$ is composed of $r^{*}$ tight geodesics with length less than $4L(-\frac{M_1}{12M_0})^{\frac{1}{2}}$ that separates $X$ in at least two non-trivial parts. 

\noindent Let us write $\Bar{N}_{g,\mu_g,0,L,r}$ for the number of tuples $\Gamma$ composed of $r$ tight closed geodesics that are all pairwise disjoint and with length less than $(-\frac{M_1}{12M_0})^{\frac{1}{2}}L$ such that $[\Gamma] \neq [\Gamma_0]$, i.e.\ $\Gamma$ is separating. This combined with the Markov inequality gives
\begin{align*}
    \mathbb{P}\bigg(\mathcal{A}_{g,L},\hat{N}_{g\mu_g,0,\varepsilon} = 0\bigg) \le\somme{r=2}{C(\varepsilon,k)}{ \mathbb{P}\bigg(\Bar{N}_{g,\mu_g,0,8L,r} \ge 1\bigg)} \le\somme{r=2}{C(\varepsilon,k)}{ \mathbb{E}[\Bar{N}_{g,\mu_g,0,8L,r} ]}.
\end{align*}
This last term can be rewritten using \eqref{integration_formula_version_generating_function} as
\begin{align}\label{bound_sep}
    \somme{r=2}{C(\varepsilon,k)}{T_{g}(\mu_g)^{-1}\somme{[\Gamma] \ne [\Gamma_0]}{}{C_{\Gamma}\inte{[0,(-\frac{M_1}{12M_0})^{\frac{1}{2}}8L]^r}{}{T_{g}(\Gamma,x,\mu_g)x_1 \cdots x_r}{x_1\cdots dx_r}}}.
\end{align}
Following the proof of the bound of \eqref{I_sep} we conclude our proof.
\end{proof}
Thus, to avoid controlling the expectation of $N^{\times}_{g,\mu_g,a_1,b_1,r_1,\dots,a_n,b_n,r_n}$, one can simply work conditionally on the event $\mathcal{A}_{g,b_n}^{c}$. However, to apply the moment method, we need to control the expectation of $\hat{N}_{g,\mu_g,a_1,b_1,r_1,\dots,a_n,b_n,r_n}$ conditionally on $\mathcal{A}_{g,b_n}^{c}$.

\begin{proposition}\label{number_of_simple_closed_tight_geodesics_conditionally_AgL}
    For any $\mu_g \underset{g \to +\infty}{\rightarrow}\mu_c$ such that $\mu_c - \mu_g = o(g^{-2})$, any disjoint compact intervals $([a_i,b_i])_{1\le i \le n}$ and any $r_1,\cdots,r_n \ge 1$ we have
    \begin{align*}
        \mathbb{E}\bigg[\hat{N}_{g,\mu_g,a_1,b_1,r_1\dots,a_n,b_n,r_n}\indi{\mathcal{A}_{g,b_n}}\bigg] \underset{g \to +\infty}{\overset{}{\longrightarrow}} 0.
    \end{align*}
    In particular
    \begin{align*}
        \mathbb{E}\bigg[\hat{N}_{g,\mu_g,a_1,b_1,r_1\dots,a_n,b_n,r_n}\bigg|\mathcal{A}_{g,b_n}^{c}\bigg] \underset{g \to +\infty}{\overset{}{\longrightarrow}} \lambda_{a_1,b_1}^{r_1}\cdots\lambda_{a_n,b_n}^{r_n}.
    \end{align*}
\end{proposition}
\begin{proof}
Fix $\varepsilon > 0$ and observe that 
\begin{align}
    \mathbb{E}\bigg[\hat{N}_{g,\mu_g,a_1,b_1,r_1\dots,a_n,b_n,r_n}\indi{\mathcal{A}_{g,b_n}}\bigg]     &\le \mathbb{E}\bigg[\hat{N}_{g,\mu_g,a_1,b_1,r_1\dots,a_n,b_n,r_n}\indi{\hat{N}_{g,\mu_g,0,\varepsilon} > 0}\bigg]  \nonumber\\&+ \mathbb{E}\bigg[\hat{N}_{g,\mu_g,a_1,b_1,r_1\dots,a_n,b_n,r_n}\indi{\mathcal{A}_{g,b_n}}\indi{\hat{N}_{g,\mu_g,0,\varepsilon} = 0}\bigg].
\end{align}
We start bounding the first term. 
Recalling the notation $r= r_1+\cdots+r_n$, we claim that we have the inequality 
\begin{align}\label{ineq1}
    \hat{N}_{g,\mu_g,a_1,b_1,r_1\dots,a_n,b_n,r_n}\indi{\hat{N}_{g,\mu_g,0,\varepsilon} > 0} \le \hat{N}_{g,\mu_g,a_1,b_1,r_1\dots,a_n,b_n,r_n,0,\varepsilon,1} +\somme{k=r}{r+3}{\bar{N}_{g,\mu_g,0,2b_n+\varepsilon,k}}.
\end{align}
Indeed, fix $X \in \mathcal{M}_{g,n}$ and $\Gamma= (\Gamma_1,\cdots,\Gamma_n)$ an ordered list where $\Gamma_i  = (\gamma_{i,1},\cdots,\gamma_{i,r_i})$ is an ordered list of disjoint simple tight closed geodesics with 
\begin{align*}
\ell(\gamma_{i,j}) \in \left[\left(-\frac{M_1}{12M_0}\right)^{\frac{1}{2}}a_i,\left(-\frac{M_1}{12M_0}\right)^{\frac{1}{2}}b_i\right]
\end{align*}
and for $i\neq j$ we have $[\Gamma_i]\cap [\Gamma_j] = \emptyset$. We also suppose that $\hat{N}_{g,\mu_g,0,\varepsilon}(X) > 0$. Let us fix the tight closed systole of $X$ and call it $\gamma$. Then the closed curve $\gamma$ is simple and $\ell(\gamma) \le \varepsilon\left(-\frac{M_1}{12M_0}\right)^{1/2}$.
We distinguish two cases:
\begin{itemize}
    \item[$\bullet$] If $\gamma$ is disjoint from $\Gamma$, we define $\varphi(X,\Gamma) = (X,\Gamma,\gamma)$.
    \item[$\bullet$]  Otherwise, if $\gamma$ intersects $\Gamma$, there exists a segment $I$ of $\gamma$ that intersects at the endpoints two distinct closed curves $\gamma_1,\gamma_2\in \Gamma$ or a closed curve $\gamma_1 \in \Gamma$. Let us call $Y = I\cup\gamma_1 \cup \gamma_2$ if we are in the first case and $Y = I \cup \gamma_1$ in the second case. In both cases, a small neighbourhood of $Y$ is a subsurface of $X$ that is either a one-holed torus, or a pair of pants. Moreover, the length of the boundaries are at most $(2b_n+\varepsilon)\left(-\frac{M_1}{12M_0}\right)^{1/2}$. We define $\Gamma'$ the boundaries of $Y$ that are not separating a disk with cusps in $X$ and not in the same extended homotopy class as any curve in $\Gamma$. In this case we write $\varphi(X,\Gamma) = (X,\Gamma,\Gamma')$. The multicurve $(\Gamma,\Gamma')$ is then a separating multicurve with all lengths at most $(2b_n+\varepsilon)\left(-\frac{M_1}{12M_0}\right)^{1/2}$. It consists of at least $r$ curves and at most $r+3$.
\end{itemize}
Since the function $\varphi$ is injective, \eqref{ineq1} follows.

For the second term we claim the inequality 
\begin{align}\label{ineq2}
\hat{N}_{g,\mu_g,a_1,b_1,r_1\dots,a_n,b_n,r_n}\indi{\mathcal{A}_{g,b_n}}\indi{\hat{N}_{g,\mu_g,0,\varepsilon} = 0}\le \somme{k=r}{r+C(\varepsilon,b_n,r)}{\bar{N}_{g,\mu_g,0,b_n(r+2),k}},
\end{align}
where $C(\varepsilon,b_n,r):=2\frac{b_n(r+2)}{\varepsilon}$.\\

Indeed, suppose $X$ satisfies $\mathcal{A}_{g,b_n}$ and $\hat{N}_{g,\mu_g,0,\varepsilon}=0$. Fix $\Gamma$ as above and let $\gamma,\beta$ be two tight closed geodesics of length less than $\left(-\frac{M_1}{12M_0}\right)^{1/2}b_n$ that intersect (or a single self-intersecting geodesic $\gamma$, for which the reasoning is the same and we let the careful reader check the details). As in the proof of Proposition~\ref{number_of_closed_geodesics_with_inter_or_non_simple_1}, let us write $\Gamma^{\times}$ for the set of curves in $\Gamma$ that intersect $\gamma$ or $\beta$. Note that we might have $\Gamma^{\times} = \emptyset$.

We consider the ribbon graph $G$ made of $\gamma,\beta$ and the curves in $\Gamma^{\times}$, where the edges of $G$ correspond to segments of closed curve between two intersections and the vertices correspond to the intersections. If $G$ has genus at most $g/2$ we define $G'=G$.
Otherwise, as in the proof of Proposition~\ref{number_of_closed_geodesics_with_inter_or_non_simple_1}, if $G$ has genus larger than $g/2$, then we can extract a connected ribbon subgraph $G' \subset G$ with genus $1 \le g' \le g/2$. Moreover, such a subgraph $G'$ can be obtained by successive removal of edges contained only in $\gamma$ or $\beta$. Indeed the curves in $\Gamma^{\times}$ are simple and do not intersect, thus if all the edges contained in $\gamma$ or $\beta$ are removed, there is no intersection anymore.

The faces of $G'$ define closed curves on $X$. The number of faces that are not filled by a disk with cusps is bounded by $C(\varepsilon,b_n,r):=2\frac{b_n(r+2)}{\varepsilon}$. Consider an arbitrary ordering $(\gamma_{r+1},\cdots,\gamma_{r+p})$ of these closed curves where we take only one copy in each extended homotopy class and omit those that share the homotopy class with a curve in $\Gamma$. Thus we have $p \le C(\varepsilon,b_n,r)$. We write $\varphi(X,\Gamma) = (X,\Gamma,\gamma_{r+1},\cdots,\gamma_{r+p})$. By construction of $(\gamma_{r+1},\cdots,\gamma_{r+p})$, the multicurve $(\Gamma,\gamma_{r+1},\cdots,\gamma_{r+p}) $ is made of disjoint simple closed geodesics and is separating. The function $\varphi$ is injective. This concludes \eqref{ineq2}. 
 
 Combining \eqref{ineq1},\eqref{ineq2}, Proposition~\ref{number_of_simple_tight_closed_geodesics} and \eqref{bound_sep}, we obtain
 \begin{align*}
     \limsup\limits_{g\to+\infty}\mathbb{E}\bigg[\hat{N}_{g,\mu_g,a_1,b_1,r_1\dots,a_n,b_n,r_n}\indi{\mathcal{A}_{g,b_n}}\bigg] \le \lambda_{a_1,b_1}^{r_1}\cdots\lambda_{a_n,b_n}^{r_n}\lambda_{0,\varepsilon}.
 \end{align*}
 Letting $\varepsilon \to 0$ concludes the proof.
\end{proof}

\subsection{Proof of Theorem~\ref{main_theorem}}

Using Proposition~\ref{number_of_closed_geodesics_with_inter_or_non_simple_1}, we only need to prove the distribution convergence of $\left(N^{\mathrm{tight}}_{g,\mu_g,a_i,b_i}\right)_{1\le i \le r}$ conditionally on $\mathcal{A}_{g,L}^{c}$ where $L=b_n$. On $\mathcal{A}_{g,L}^{c}$, for any $r_1,\cdots,r_n\ge 1$ we have 
\begin{center}
$(N^{\mathrm{tight}}_{g,\mu_g,a_1,b_1})_{r_1}\cdots(N^{\mathrm{tight}}_{g,\mu_g,a_n,b_n})_{r_n} = \hat{N}_{g,\mu_g,a_1,b_1,r_1\dots,a_n,b_n,r_n}$. 
\end{center}
Then the proof follows from the moment method using Proposition~\ref{number_of_simple_closed_tight_geodesics_conditionally_AgL}.
\qed

\subsection{Proofs of corollaries}

\begin{proof}[Proof of Corollary~\ref{main_corollary}]
Fix $(\mu_g)_{g \ge 0}$ such that $\mu_c - \mu_g \underset{}{\sim}\displaystyle \frac{5g}{2 n_g}\mu_c$. We have $\mu_c - \mu_g = o(g^{-2})$. Thus we can apply Proposition~\ref{number_of_cusps_mean_and_concentration} and Proposition~\ref{Concentration_result} and obtain  
\begin{align*}
    \mathbb{E}[\mathcal{N}_{g,\mu_g}] \underset{g \to +\infty}{\sim} n_g \text{ and }\frac{\mathcal{N}_{g,\mu_g}}{n_g} \underset{g \to +\infty}{\overset{(\mathbb{P})}{\to}}1.
\end{align*}
Recall the constants $\alpha_1 = \sqrt{\frac{6}{\pi}\sqrt{\frac{2j_0}{J_1(j_0)}}}$ and $\alpha_2 = \frac{1}{\pi}\sqrt{3j_0\sqrt{5}}$. We find with our choice of $(\mu_g)_{g \ge 0}$:
\begin{align*}
    \alpha_1 (\mu_c-\mu_g)^{\frac{1}{4}} \underset{g \to +\infty}{\sim} \alpha_2\bigg(\frac{n_g}{g}\bigg)^{-\frac{1}{4}}.
\end{align*}
Combining this asymptotic equivalence with Theorem~\ref{main_theorem}, we conclude the proof.
\end{proof}
\noindent
Finally, we deduce the laws of the non-separating and tight systoles.
\begin{proof}[Proof of Corollary~\ref{systole}]
    Fix $(n_g)_{g \ge 0}$ such that $\displaystyle \frac{n_g}{g^3} \underset{g \to +\infty}{\to}+\infty$. The choice of 
    $\mu_g$ is as in last proof. Observe that $\Pf^{\mathrm{WP}}_{g,\mu_g}$-almost surely $\ell^{\mathrm{tight}}_{sys} \le \ell^{\,\mathrm{ns}}_{\mathrm{sys}}$. Indeed, $\Pf^{\mathrm{WP}}_{g,\mu_g}$-almost surely, to any non-separating geodesic $\gamma$ on $X$, we can associate a unique tight geodesic $\gamma^{'}$ in its extended homotopy class. By definition it gives the desired bound. On the event $\{N^{\times}_{g,\mu_g,0,t} = 0\} \cap \{N_{g,\mu_g,0,t} \neq 0\}$, we have $\ell^{\mathrm{tight}}_{sys} = \ell^{\,\mathrm{ns}}_{\mathrm{sys}}$. Thus we only have to bound
    \begin{align}
        \Pf\bigg(\{N^{\times}_{g,\mu_g,0,t} \ge 1\} \cup \{N_{g,\mu_g,0,t} = 0\}\bigg).
    \end{align}
Combining Proposition~\ref{number_of_closed_geodesics_with_inter_or_non_simple_1} and Theorem~\ref{main_theorem}, we deduce that for any $t \ge 0$:
    \begin{align}
        \limsup_{g} \Pf\bigg(\{N^{\times}_{g,\mu_g,0,t} \ge 1\} \cup \{N_{g,\mu_g,0,t} = 0\}\bigg) \le \exp(-\lambda_{0,t}).
    \end{align}
    It follows that
    \begin{align*}
         \limsup_{g} \Pf^{\mathrm{WP}}_{g,\mu_g}\bigg(\ell^{\mathrm{tight}}_{sys} \neq \ell^{\,\mathrm{ns}}_{\mathrm{sys}}\bigg) \le \exp(-\lambda_{0,t}).
    \end{align*}
    Letting $t \to +\infty$, we deduce
    \begin{align*}
         \lim_{g \to +\infty}\Pf^{\mathrm{WP}}_{g,\mu_g}\bigg(\ell^{\mathrm{tight}}_{sys} \neq \ell^{\,\mathrm{ns}}_{\mathrm{sys}}\bigg) = 0.
    \end{align*}
Combining this and Corollary~\ref{main_corollary} completes the proof for the distribution convergence.
\end{proof}

\printbibliography
\end{document}